\documentclass[12pt]{amsart}
\usepackage{amssymb, amsmath, mathtools}

\usepackage[utf8]{inputenc}
\usepackage[T1]{fontenc}

\usepackage{fullpage}

\usepackage{amsmath,amscd,mathrsfs}

\usepackage{tikz}

\newtheorem{Theorem}{Theorem}[section]
\newtheorem{Lemma}[Theorem]{Lemma}
\newtheorem{Corollary}[Theorem]{Corollary}
\newtheorem{Proposition}[Theorem]{Proposition}

\newtheorem{Remark}[Theorem]{Remark}

\def\height{\mbox{ht}\,}
\def\heigh{\tiny{\mbox{ht}}\,}
\def\Spec{\mbox{Spec}\,}

\def\q{\mathfrak{q}}
\def\m{\mathfrak{m}}

\def\e{e_{HK}}
\def\F{F^e_*}

\def\R{\mathbb{R}}

\def\N{\mathbb{N}}
\def\CC{\mathscr{C}}
\def\DD{\mathscr{D}}

\def\Supp{\mbox{Supp}\,}
\def\Ann{\mbox{Ann}\,}

\def\Hom{\mbox{Hom}}

\title{Uniform Bounds in F-finite Rings and Lower Semi-Continuity of the F-Signature}
\author{Thomas Polstra } 
\address{Thomas Polstra, Department of Mathematics, University of Missouri-Columbia, Columbia, MO 65211}
\email{tmpxv3@mail.missouri.edu}

\begin{document}

\let\thefootnote\relax\footnote{2010 \emph{Mathematics Subject Classification.} 13A35, 13D40, 13F40, 14B05}

\maketitle

\begin{abstract} This paper establishes uniform bounds in characteristic $p$ rings which are either F-finite or essentially of finite type over an excellent local ring. These uniform bounds are then used to show that  the Hilbert-Kunz length functions and the normalized Frobenius splitting numbers defined on the Spectrum of a ring converge uniformly to their limits, namely the Hilbert-Kunz multiplicity function and the F-signature function. From this we establish that the F-signature function is lower semi-continuous. Lower semi-continuity of the F-signature of a pair is also established. We also give a new proof of the upper semi-continuity of Hilbert-Kunz multiplicity, which was originally proven by Ilya Smirnov.

\end{abstract}

\section{Introduction}

Throughout this paper all rings are assumed to be commutative, Noetherian, with identity, and of prime characteristic $p$. We shall reserve $q$ to denote a power of $p$, i.e., $q=p^e$ for some nonnegative integer $e$, and $\lambda(-)$ denotes the standard length function.

\medskip

If $(R,\m)$ is a local ring of dimension $d$, $M$ a finitely generated $R$-module, and $I$ an $\m$-primary ideal of $R$, then \emph{the $q$th Hilbert-Kunz length of $M$ at $I$} is given by $\frac{1}{q^{\dim(M)}}\lambda(M/I^{[q]}M)$. \emph{The Hilbert-Kunz multiplicity of $M$ at $I$} is defined by $$\e(I, M):=\lim_{q\rightarrow \infty}\frac{1}{q^{\dim(M)}}\lambda\left(\frac{M}{I^{[q]}M}\right).$$ Paul Monsky showed this limit always exists in \cite{Monsky1983}.

\medskip

We say that a ring $R$ is F-finite if the Frobenius endomorphism $F: R\rightarrow R$ which maps $r\mapsto r^p$ makes $R$ a finite $R$-module. This is equivalent to $R$ being module finite over $F^{e}(R)$ for all $e\geq 1$, and when $R$ is reduced this is equivalent to $R^{1/q}$ being module finite over $R$ for all $q$. If $R$ is an F-finite ring, then so is any finitely generated algebra and localization over $R$. If $(R,\m,k)$ is local then we let $\alpha(R)=\log_p[k^{1/p}:k]$. If $R$ is not necessarily local and $P\in \Spec(R)$, then we let $\alpha(P)=\alpha(R_P)$. 

\medskip

In this paper we will be interested in uniform properties of Hilbert-Kunz functions over F-finite rings and rings essentially of finite type over an excellent local ring. For the sake of simplicity, assume that $I=\m$ and $M=R$. Then not only does $\frac{1}{q^d}\lambda(R/\m^{[q]})$ converge to $\e(R):=\e(\m,R)$, it is the case that $\lambda(R/\m^{[q]}R)=\e(R)q^d+O(q^{d-1})$, hence there exists a $C>0$ such that for all $q$, $\lambda(R/\m^{[q]}R)\leq Cq^d$. Now suppose that $R$ is a not necessarily local characteristic $p$ ring. Then for each $P\in \Spec(R)$ there exits a constant $C>0$ such that for each $q$, $\lambda(R_P/P^{[q]}R_P)\leq Cq^{\heigh(P)}$. This particular result is not very interesting since the constant $C$ depends upon $P$ and is easily obtained from well known results about Hilbert-Kunz length functions. Recently, I. Smirnov showed, if $R$ is an excellent ring of characteristic $p$, that for each $P\in\Spec(R)$ there exists a constant $C$ and element $s\in R-P$ such that for all $Q\in D(s)\cap V(P)$, $\lambda_{R_Q}(R_Q/Q^{[q]}R_Q)\leq Cq^{\heigh(Q)}$ (see Lemma 14 in \cite{Smirnov2014}). We significantly improve this result in Proposition \ref{Main Lemma} and Proposition \ref{Main Lemma Finite Type} for rings which are either F-finite or are essentially of finite type over an excellent local ring, both of which are large classes of excellent rings. A consequence of Proposition \ref{Main Lemma} and Proposition \ref{Main Lemma Finite Type} is that, if $R$ is F-finite or essentially of finite type over an excellent local ring, then there exists a constant $C$ such that for all $P\in\Spec(R)$, $\lambda(R_P/P^{[q]}R_P)\leq Cq^{\heigh(P)}$. 

\medskip

A \emph{map of primary ideals} is a map $I(-):\Spec(R)\rightarrow \{\mbox{Ideals of }R\}$ such that for each $P\in \Spec(R)$, $I(P)R_P$ is a $PR_P$-primary. If $M$ is a finitely generated $R$-module then we shall denote by $\lambda_{q_1}^M(I(-))$, or simply $\lambda_{q_1}(I(-))$, if $M$ is understood, to be the function from the support of $M$, which we denote $\Supp(M)$, to the real numbers $\R$, which maps a prime $P\mapsto \frac{1}{q_1^{\dim(M_P)}}\lambda(M_P/I(P)^{[q_1]}M_P)$. We denote by $\e(I(-), M_{-})$ the function which sends a prime $P\in\Supp(M)$ to $\e(I(P), M_P)$.

\medskip

Let $R$, $M$, and $I(-)$ be as above. Then it is easy to see that $\lambda_{q_1}(I(-))$ converges pointwise to $\e(I(-), M_{-})$ as $q_1\rightarrow \infty$. Theorem \ref{Main Theorem} states that if $R$ is an F-finite ring or essentially of finite type over an excellent local ring, $M$ is a finitely generated $R$-module, and $I(-)$ a map of primary ideals, then $\lambda_{q_1}^M(I(-))/\lambda^M_1(I(-))$ converges uniformly to $\e(I(-), M_{-})/\lambda^M_1(I(-))$ as $q_1\rightarrow \infty$. In particular, if $M=R$ and $I(P)=P$, then the $q$th Hilbert-Kunz function, which sends a prime $P\mapsto \lambda(R_P/P^{[q]}R_P)$, converges uniformly to the Hilbert-Kunz multiplicity function, which sends a prime $P\mapsto \e(R_P)$. In order to prove this we will need to establish some uniform bounds in F-finite rings and in rings essentially of finite type over an excellent local ring. Some of the uniform bounds established in Section \ref{Uniform Bounds} and Section \ref{Uniform Bounds 2} of this paper are related to, but are often improvements of, uniform bounds in Section 3 of \cite{Smirnov2014}, which establishes the upper semi-continuity of Hilbert-Kunz multiplicity, and Section 3 of \cite{Tucker2012}, which shows that the F-signature of a local ring exists.


\medskip

 I. Smirnov has recently shown in \cite{Smirnov2014} that $\e(R_{-})$ is upper semi-continuous on locally equidimensional rings which are F-finite or essentially of finite type over an excellent local ring. In showing that the Hilbert-Kunz multiplicity function is the uniform limit of upper semi-continuous functions on such rings, we easily recover Smirnov's result. It is still unknown if Hilbert-Kunz multiplicity is upper semi-continuous on an excellent locally equidimensional ring.
 
 \medskip
 
 Another interesting invariant defined on a local ring $(R,\m)$ of characteristic $p$ is the F-signature of $R$, defined originally in \cite{HuLeu2002}, by Huneke and Leuschke. For any module $M$ and any $q=p^e$, we can view $M$ as an $R$-module via restriction of scalars under $F^e$ which we denote by $F^q_* M$. In particular if $R$ is F-finite and $M=R$, then $F^q_*R$ is a module finite $R$-module and we can write $F^q_*R \simeq R^{a_q}\oplus M_q$ where $M_q$ has no free $R$-summand. The number $a_q$ is called the $q$th Frobenius splitting number of $R$. We denote by $b_q=\frac{a_q}{q^{\alpha(R)}}$. The number $s_q:=\frac{a_q}{q^{\alpha(R)+\dim(R)}}$ is called the $q$th normalized Frobenius splitting number of $R$. Yao showed there is a way to measure $b_q$, in way that is well defined even for rings which are not F-finite, hence one can define the $q$th normalized Frobenius splitting number for a ring which is not F-finite (\cite{Yao2006}, Lemma 2.1). Huneke and Leuschke defined the F-signature of a local ring of dimension $d$ to be the limit $\lim_{q\rightarrow} \frac{b_q}{q^d}$, provided the limit exists. Kevin Tucker showed in \cite{Tucker2012} that the F-signature of a local ring always exists.
 
 \medskip
 
 If $R$ is an F-finite ring, which is not necessarily local, then define the $q$th Frobenius splitting number function $a_q:\Spec(R)\rightarrow \R$ by letting $a_q(P)$ be the $q$th Frobenius splitting number of the local ring $R_P$. If $R$ is any characteristic $p$ ring, then define the $q$th normalized Frobenius splitting number function $s_q:\Spec(R)\rightarrow \R$ by letting $s_q(P)$ be the $q$th normalized Frobenius splitting number of the local ring $R_P$. We let $b_q(P)=q^{\heigh(P)}s_q(P)$ and we let $s:\Spec(R)\rightarrow \R$ be the F-signature function which sends a prime $P\mapsto s(R_P)$, the F-signature of the local ring $R_P$. 
 
 \medskip
 
 The problem of whether the F-signature function is a lower semi-continuous function with respect to the Zariski topology has been of interest for quite some time. Recall that a function $f: X\rightarrow \R$, where $X$ is a topological space, is lower semi-continuous at $x\in X$ if for all $\epsilon> 0$, there is an open neighborhood $U$ of $x$ such that $f(x)-f(y)<\epsilon$ for all $y\in U$. In other words, a function $f$ is lower semi-continuous at $x$ if in a small enough open neighborhood of $x$ the numbers $f(y)$ as $y$ varies in the open neighborhood of $x$ can only be slightly smaller than $f(x)$. We would like to briefly explain why it has been suspected that the F-signature function should satisfy this property.
 
  \medskip
 
The F-signature detects subtle information about the severity of the singularity of a local ring. Given a local ring $(R,\m)$, it always the case that $0\leq s(R)\leq 1$. Huneke and Leuschke showed in \cite{HuLeu2002} that $s(R)=1$ if and only if $R$ is a regular local ring. Aberbach and Leuschke showed in \cite{AberbachLeuschke2003} that $s(R)>0$ if and only if $R$ is strongly F-regular. Heuristically, the closer to $1$ the F-signature of $R$ is the "nicer" the singularity is, and the closer to $0$ the "worse" the singularity is. One expects that given a ring or a scheme with decent geometric properties, that the severity of a singularity of a point is controlled in an open neighborhood of that point.  This is exactly what we should expect the F-signature function defined on the spectrum of a "decent" ring, e.g. and excellent domain, to do. Given a prime $P$ in the spectrum of such a ring, we expect that in a small enough open neighborhood of the prime, that the singularities found in that open neighborhood are not too much worse then the singularity associated with $P$. Thus we should expect than in a small enough open neighborhood $U$ of $P$ that $s(Q)$ is at most $\epsilon$ closer to $0$, which is precisely what lower semi-continuity of the F-signature would say.
 
 \medskip
 
  Another reason to expect the F-signature function to be lower semi-continuous is that Enescu and Yao showed in \cite{EnescuYao2011} that under mild conditions, the $q$th normalized Frobenius splitting number function is a lower semicontinuous functions. For example, they showed that if $R$ is a domain which is either F-finite or essentially of finite type over an excellent local ring, then the $q$th normalized Frobenius number function is lower semi-continuous. So after Kevin Tucker showed the F-signature to always exist, it has been known that the F-signature function naturally arrises as the limit of lower semi-continuous functions. 
  
  \medskip
  
  Some light has been previously shed on the lower semi-continuity of the F-signature problem. Blickle, Schwede, and Tucker showed that if $R$ is a regular and not necessarily local F-finite ring with $0\neq f\in R$ and $t\geq 0$, then the function $\Spec(R)\rightarrow \R$ defined by $P\mapsto s(R_P, f^t)$ is lower semi-continuous. See \cite{BST2013} for more details. The results in section \ref{Lower Semi-Continuity of F-signature of Pairs} in this paper recapture Blickle, Schewede, and Tucker's result.

 \medskip
 
 Theorem \ref{F-signature 8} proves that if $R$ is either F-finite or is essentially of finite type over an excellent local ring, then the $q$th normalized Frobenius splitting number functions converge uniformly to the F-signature function as $q\rightarrow \infty$. It will then follow by Enescu's and Yao's work, \cite{EnescuYao2011}, that the F-signature function will be lower semi-continuous on all such rings. Kevin Tucker has independently found, and discussed with the author, an alternative proof of the lower semi-continuity of the F-signature.

 \medskip
 
 The paper is organized as follows. In section \ref{Preliminary} we establish some preliminary results. In particular, section \ref{Preliminary} contains a generalized version of a Lemma of Sankar Dutta which is crucial to the bounds given in section \ref{Uniform Bounds}. Section \ref{Uniform Bounds} establishes uniform bounds of Hilbert-Kunz length functions in F-finite rings. Section \ref{Uniform Bounds} is the most difficult section to work through, but the bounds that are established lead to a proof that the F-signature function is lower semi-continuous on F-finite rings and rings essentially of finite type over an excellent local ring. Section \ref{Uniform Bounds 2} establishes the bounds in \ref{Uniform Bounds} for rings which are essentially of finite type over an excellent local ring. In Section \ref{Applications} we apply the results of section \ref{Uniform Bounds} and \ref{Uniform Bounds 2} to establish the uniform convergence of Hilbert-Kunz length functions and normalized Frobenius splitting numbers to their limits. In Section \ref{Lower Semi-Continuity of F-signature of Pairs} lower semi-continuity of the F-signature of a pair, $(R, \mathscr{D})$, is established for all Cartier subalgebras $\mathscr{D}$ on an F-finite ring $R$.

\section{Preliminary Results}\label{Preliminary}



If $R$ is an F-finite ring which is locally equidimensional, then it was originally shown by E. Kunz in \cite{Kunz1976} that the function $\Spec(R)\rightarrow \R$ which sends $P\mapsto \alpha(P)+\height(P)$ is constant on connected components of $\Spec(R)$. In particular, if $R$ is an F-finite domain then $\alpha(P)+\height(P)$ is constant on $\Spec(R)$. If $R$ is an F-finite domain then we let $\gamma(R)$ be the constant $\alpha(P)+\height(P)$.

\medskip

We will need a global version of a Lemma, first proved by Sankar Dutta, in order to establish the uniform bounds found in Section \ref{Uniform Bounds} and \ref{Uniform Bounds 2}. In \cite{Dutta1983}, Sankar Dutta showed that if $(R,\m)$ is an F-finite local domain of dimension $d$ then there exists a finite set of nonzero primes $\mathcal{S}(R)$ and a constant $C$ such that for all $q=p^e$ there is a containment of $R$-modules $R^{q^{\gamma(R)}}\subseteq R^{1/q}$ which has a prime filtration whose prime factors are isomorphic to $R/P$, where $P\in\mathcal{S}(R)$, and such a prime factor appears no more than $Cq^{\gamma(R)}$ times in the filtration. In particular, the length of the prime filtration of $R^{q^{\gamma(R)}}\subseteq R^{1/q}$ has length no more than $C|\mathcal{S}(R)|q^{\gamma(R)}$. This result, for local domains whose residue field is perfect, is exercise 10.4 in \cite{Huneke1996}, whose proof is given in the second appendix by Karen Smith, and this result is explicitly stated and proved in \cite{Huneke2013} as Lemma 4. 

\begin{Remark}\label{Remark 1}\emph{If $R$ is an F-finite domain and $P\in \Spec(R)$ a nonzero prime, then $\gamma(R/P)=\log_p[\frac{R_P}{PR_P}^{1/p}: \frac{R_P}{PR_P}]=\alpha(P)< \alpha(P)+\height(P)=\gamma(R)$.}

\end{Remark}



\begin{Lemma}\label{Global version of Dutta's Lemma} Let $R$ be an F-finite domain. Then there exists a finite set of nonzero primes $\mathcal{S}(R)$, and a constant $C$, such that for every $q=p^e$, 
\begin{enumerate}

\item there is a containment of $R$-modules $R^{q^{\gamma(R)}}\subseteq R^{1/q}$, 

\item which has a prime filtration whose prime factors are isomorphic to $R/P$, where $P\in\mathcal{S}(R)$,

\item and for each $P\in \mathcal{S}(R)$, the prime factor $R/P$ appears no more than $Cq^{\gamma(R)}$ times in the prime filtration of the containment $R^{q^{\gamma(R)}}\subseteq R^{1/q}$.

\end{enumerate}

\end{Lemma}

\begin{proof} We shall prove the statement by induction on the Krull dimension of $R$. If the dimension $R$ is $0$, then $R$ is a field and the Lemma is trivial.

\medskip

Now suppose that $\dim(R)>0$. Then $R^{1/p}$ is a torsion-free $R$-module of rank $p^{\gamma(R)}$. Hence there is an injection of $R$-modules $R^{p^{\gamma(R)}}\subseteq R^{1/p}$ so that the support of the cokernel $R^{1/p}/R^{p^{\gamma(R)}}$ consists of nonzero primes. Therefore $R^{p^{\gamma(R)}}\subseteq R^{1/p}$ has a prime filtration of the following form with the quotients $M_i/M_{i-1}= R/P_i$ where $P_i$ is a nonminimal prime of $R$, $$R^{p^{\gamma(R)}}=M_0\subseteq M_1\subseteq \cdots \subseteq M_h=R^{1/p}.$$ The quotients $R/P_i$ are F-finite domains of smaller Krull dimension than $R$, and so we may assume by induction that the result holds for each $R/P_i$, with finite collection of primes $\mathcal{S}(R/P_i)$ and constant $C_i$. Let $C'=\sum C_i$ and $\mathcal{S}(R)=\bigcup (\mathcal{S}(R/P_i)\cup \{P_i\})$. Observe that the above filtration shows that $R^{p^{\gamma(R)}}\subseteq R^{1/p}$ has a prime filtration consisting of no more than $C'$ quotients isomorphic to $R/P$ for each $P\in \mathcal{S}(R)$ and all prime factors are isomorphic to $R/P$ for some $P\in\mathcal{S}(R)$. We shall show by induction that $R^{q^{\gamma(R)}}\subseteq R^{1/q}$ has a prime filtration whose prime factors are isomorphic to $R/P$ with $P\in\mathcal{S}(R)$ with no more than $C'q^{\gamma(R)}(1+\frac{1}{p}+\frac{1}{p^2}+\cdots+\frac{1}{q})$ quotients isomorphic to $R/P$ for each $P\in\mathcal{S}(R)$.

\medskip

Now suppose that $R^{q^{\gamma(R)}}=N_0\subseteq N_1\subseteq \cdots \subseteq N_m= R^{1/q}$ is a prime filtration of $R^{q^{\gamma(R)}}\subseteq R^{1/q}$ whose prime factors are isomorphic to $R/P$ with $P\in\mathcal{S}(R)$ with no more than $C'q^{\gamma(R)}(1+\frac{1}{p}+\cdots +\frac{1}{q})$ quotients isomorphic to $R/P$ for each $P\in \mathcal{S}(R)$. Take $q$th roots of the modules in the filtration $R^{p^{\gamma(R)}}\subseteq R^{1/p}$ to get the following new filtration, $$(R^{1/q})^{p^{\gamma(R)}}=M_0^{1/q}\subseteq M_1^{1/q}\subseteq \cdots \subseteq M_h^{1/q}=R^{1/pq}.$$ Each of the quotients $M_i^{1/q}/M_{i+1}^{1/q}=(R/P_i)^{1/q}$. By induction there exists a prime filtration of $M_{i-1}^{1/q}\subseteq M_i^{1/q}$ with precisely $q^{\gamma(R)}$ prime factors isomorphic to $R/P_i$ and each other prime factor is isomorphic to $R/P$ for some $P\in \mathcal{S}(R/P_i)$ and such a prime factor appears no more than $C_iq^{\gamma(R/P_i)}$ times in the filtration. Furthermore, the prime filtration $R^{q^{\gamma(R)}}=N_0\subseteq N_1\subseteq \cdots \subseteq N_m= R^{1/q}$ gives the following filtration of $(R^{q^{\gamma(R)}})^{p^{\gamma(R)}}=R^{(pq)^{\gamma(R)}}\subseteq (R^{1/q})^{p^{\gamma(R)}}$, $$R^{(pq)^{\gamma(R)}}=N_0^{p^{\gamma(R)}}\subseteq N_1^{p^{\gamma(R)}}\subseteq \cdots \subseteq N_m^{p^{\gamma(R)}}=(R^{1/q})^{p^{\gamma(R)}}.$$ Hence $R^{(pq)^{\gamma(R)}}\subseteq (R^{1/q})^{p^{\gamma(R)}}$ has a prime filtration with prime factors isomorphic to $R/P$ with $P\in\mathcal{S}(R)$ and such a prime factor appears no more than $C'(pq)^{\gamma(R)}(1+\frac{1}{p}+\cdots +\frac{1}{q})$ times in the filtration.

\medskip

 Putting the above information together we get that there is an embedding of $R^{(pq)^{\gamma(R)}}\subseteq R^{1/pq}$ with prime filtration whose prime factors are isomorphic to $R/P$ with $P\in\mathcal{S}(R)$ and there are no more than the following number of quotients isomorphic to $R/P$ for each $P\in \mathcal{S}(R)$:
\begin{eqnarray*}
C'(pq)^{\gamma(R)}\left(1+\frac{1}{p}+\cdots +\frac{1}{q}\right)&+&\sum_{i=1}^hC_iq^{\gamma(R/P_i)}.
\end{eqnarray*}
By Remark \ref{Remark 1} we know that each $\gamma(R/P_i)\leq \gamma(R)-1$, and so we have the following estimates,
\begin{eqnarray*}
C'(pq)^{\gamma(R)}\left(1+\frac{1}{p}+\cdots +\frac{1}{q}\right)&+&\sum_{i=1}^hC_iq^{\gamma(R/P_i)}\\\\
&\leq&C'(pq)^{\gamma(R)}\left(1+\frac{1}{p}+\cdots +\frac{1}{q}\right)+\sum_{i=1}^hC_iq^{\gamma(R)-1}\\\\
&=&C'(pq)^{\gamma(R)}\left(1+\frac{1}{p}+\cdots +\frac{1}{q}\right)+C'q^{\gamma(R)-1}\\\\
&=&C'(pq)^{\gamma(R)}\left(1+\frac{1}{p}+\cdots \frac{1}{q}+\frac{1}{p^{\gamma(R)}q}\right)\\\\
&\leq&C'(pq)^{\gamma(R)}\left(1+\frac{1}{p}+\cdots +\frac{1}{q}+\frac{1}{pq}\right).
\end{eqnarray*}

Each of the sums $1+\frac{1}{p}+\cdots +\frac{1}{q}\leq 1+\frac{1}{2}+\cdots + \frac{1}{2^e}\leq 2$. It now follows by induction that for every $q$, that the containment $R^{q^{\gamma(R)}}\subseteq R^{1/q}$ will have a prime filtration whose factors are isomorphic to $R/P$ for some $P\in\mathcal{S}(R)$ with no more than $C'(1+\frac{1}{p}+\cdots +\frac{1}{q})q^{\gamma(R)}\leq 2C'q^{\gamma(R)}$ quotients isomorphic to $R/P$ for each $P\in\mathcal{S}(R)$.

\end{proof}

We will find it useful in Section \ref{Uniform Bounds} to have a version of Dutta's Lemma with the inclusion of $R^{q^{\gamma(R)}}\subseteq R^{1/q}$ reversed.

\begin{Corollary}\label{Global version of Dutta's Lemma part 2} Let $R$ be an F-finite domain. Then there exists a finite set of nonzero primes $\mathcal{S}(R)$, and a constant $C$, such that for every $q=p^e$, 
\begin{enumerate}

\item there is a containment of $R$-modules $R^{1/q}\subseteq R^{q^{\gamma(R)}}$, 

\item which has a prime filtration whose prime factors are isomorphic to $R/P$, where $P\in\mathcal{S}(R)$,

\item and for each $P\in \mathcal{S}(R)$, the prime factor $R/P$ appears no more than $Cq^{\gamma(R)}$ times in the prime filtration of the containment $R^{1/q}\subseteq R^{q^{\gamma(R)}}$.

\end{enumerate}
\end{Corollary}

\begin{proof} Since $R^{1/p}$ is torsion-free of rank $p^{\gamma(R)}$, there is an injection of $R$-modules $R^{1/p}\subseteq R^{p^{\gamma(R)}}$ so that the support of the cokernel $R^{p^{\gamma(R)}}/R^{1/p}$ consists of nonzero primes.  Therefore there is prime filtration $R^{1/p}=M_0\subseteq M_1\subseteq \cdots \subseteq M_h=R^{p^{\gamma(R)}}$ with $M_i/M_{i-1}\simeq R/P_i$ with $P_i$ a nonzero prime ideal. By Remark \ref{Remark 1} $\gamma(R/P_i)<\gamma(R)$. We let $\mathcal{S}(R/P_i)$ and constant $C_i$ be the collection of primes and constant as described in Lemma \ref{Global version of Dutta's Lemma} for the F-finite domain $R/P_i$. As in the proof of Lemma \ref{Global version of Dutta's Lemma} we let $C'=\sum C_i$ and $\mathcal{S}(M)=\bigcup (\mathcal{S}(R/P_i)\cup \{P_i\})$. Furthermore, once again as in Lemma \ref{Global version of Dutta's Lemma}, we can show by induction that $R^{1/q}\subseteq R^{q^{\gamma(R)}}$ has a prime filtration whose prime factors are isomorphic to $R/P$ with $P\in\mathcal{S}(M)$ with no more than $C'q^{\gamma(R)}(1+\frac{1}{p}+\frac{1}{p^2}+\cdots+\frac{1}{q})$ quotients isomorphic to $R/P$ for each $P\in\mathcal{S}(R)$. The above filtration of $R^{1/p}\subseteq R^{p^{\gamma(R)}}$ shows the induction step when $q=p$. 

\medskip

Now suppose that $R^{1/q}=N_0\subseteq N_1\subseteq \cdots \subseteq N_m=R^{q^{\gamma(R)}}$ is a prime filtration of $R^{1/q}\subseteq R^{q^{\gamma(R)}}$ such that each $N_j/N_{i-j}\simeq R/P_j$ for some $P_j\in \mathcal{S}(R)$ and such a prime factor appears no more than $C'q^{\gamma(R)}(1+\frac{1}{p}+\frac{1}{p^2}+\cdots+\frac{1}{q})$ times in the filtration. Therefore $(R^{p^{\gamma(R)}})^{1/q}=(R^{1/q})^{p^{\gamma(R)}}\subseteq (R^{q^{\gamma(R)}})^{p^{\gamma(R)}}=R^{(qp)^{\gamma(R)}}$ has a prime filtration with prime factors $R/P_{j}$ with $P_j\in \mathcal{S}(R)$ and such a prime factor appears no mare than $C'(qp)^{\gamma(R)}(1+\frac{1}{p}+\frac{1}{p^2}+\cdots+\frac{1}{q})$ times in the filtration. Furthermore, the prime filtration $R^{1/p}=M_0\subseteq M_1\subseteq\cdots \subseteq M_n= R^{p^{\gamma(R)}}$ gives the following filtration of $R^{1/pq}=(R^{1/p})^{1/q}\subseteq (R^{p^{\gamma(R)}})^{1/q}$, $$(R^{1/p})^{1/q}=M_0^{1/q}\subseteq M_{1}^{1/q}\subseteq \cdots \subseteq M_n^{1/q}=(R^{p^{\gamma(R)}})^{1/q}.$$ Since $M_i^{1/q}/M_{i-1}^{1/q}\simeq (R/P_i)^{1/q}$, we apply Lemma \ref{Global version of Dutta's Lemma} to know there is a prime filtration of each $M^{1/q}_{i-1}\subseteq M^{1/q}_i$ whose prime factors come from $\mathcal{S}(R/P_i)$ and such a prime factor appears no more than $C_iq^{\gamma(R/P_i)}\leq C'q^{\gamma(R)-1}$ times in the filtration. Putting all of this information together we get an embedding $R^{1/pq}\subseteq R^{(pq)^{\gamma(R)}}$ with a prime filtration whose prime factors come from $\mathcal{S}(R)$ and such a prime factor appears no more than the following number in the filtration,  $$p^{\gamma(R)}C'q^{\gamma(R)}\left(1+\frac{1}{p}+\cdots +\frac{1}{q}\right)+\sum_{i=1}^hC_iq^{\gamma(R/P_i)}\leq C'(qp)^{\gamma(R)}(1+\frac{1}{p}+\cdots +\frac{1}{pq})\leq 2C'(qp)^{\gamma(R)}.$$

\end{proof}


We combine of Lemma \ref{Global version of Dutta's Lemma} and Corollary \ref{Global version of Dutta's Lemma part 2} into a single statement for convenience.

\begin{Corollary}\label{Global version of Dutta's Lemma part 3} Let $R$ be an F-finite domain. There exists a finite set of nonzero primes $\mathcal{S}(R)$ and a constant $C$ such that for every $q=p^e$, there is a containment of $R$-modules $R^{1/q}\subseteq R^{q^{\gamma(R)}}$ and $R^{q^{\gamma(R)}}\subseteq R^{1/q}$ which each has a prime filtration whose prime factors are isomorphic to $R/P$, where $P\in \mathcal{S}(R)$, and such a prime factor appears no more than $Cq^{\gamma(R)}$ times in the filtration. 
\end{Corollary}

We shall need the following two well know lemmas, whose proofs are given for the sake of completion.

\begin{Lemma}\label{Length Lemma} Let $(R,\m,k)$ be an F-finite reduced local ring and let $I$ be an $\m$-primary ideal. Then $\lambda(R^{1/q}/IR^{1/q})=q^{\alpha(R)}\lambda(R/I^{[q]}R)$.

\end{Lemma}

\begin{proof} Consider a prime filtration of $I^{[q]}R\subseteq R$, say it is given by $I^{[q]}R=M_0\subseteq M_1\subseteq \cdots \subseteq M_n=R$ with each $M_i/M_{i-1}\simeq k$. Then by taking $q$th roots we get a filtration $IR^{1/q}=M_0^{1/q}\subseteq M_1^{1/q}\subseteq \cdots \subseteq M_n^{1/q}=R^{1/q}$ with each quotient $M_i^{1/q}/M_{i-1}^{1/q}\simeq k^{1/q}$. It follows that $\lambda(R^{1/q}/IR^{1/q})=q^{\alpha(R)}\lambda(R/I^{[q]}R)$.

\end{proof}

\begin{Lemma}\label{Multiplicity Lemma} Let $(R,\m, k)$ be a local characteristic $p$ ring, $I$ be an $\m$-primary ideal, and $M$ a finitely generated $R$-module. Then $$\lim_{q_2\rightarrow \infty}\frac{1}{q_2^{\dim(M)}}\lambda(M/I^{[q_1q_2]}M)=q_1^{\dim(M)}\e(I,M).$$

\end{Lemma}

\begin{proof} We only have to observe that
\begin{eqnarray*}
\lim_{q_2\rightarrow \infty}\frac{1}{q_2^{\dim(M)}}\lambda(M/I^{[q_1q_2]}M)&=&\lim_{q_2\rightarrow \infty}\frac{q_1^{\dim(M)}}{(q_1q_2)^{\dim(M)}}\lambda(M/I^{[q_1q_2]}M)\\ \\
&=&q_1^{\dim(M)}\lim_{q\rightarrow \infty}\frac{1}{q^{\dim(M)}}\lambda(M/I^{[q]}M)\\ \\
&=&q_1^{\dim(M)}\e(I,M).
\end{eqnarray*}

\end{proof}

\section{Uniform Bounds in F-Finite Rings}\label{Uniform Bounds}

The goal of this section is to establish uniform bounds in not necessarily local F-finite rings. The purpose of establishing these uniform bounds is to better understand the global behavior of relative Hilbert-Kunz length functions, which can then be used to establish the lower semi-continuity of the F-signature. 

\begin{Remark}\label{Remark} \emph{If $(R,\m)$ is a local ring, an \emph{$\m$-primary pair of ideals} will  be a containment of ideals of $R$, $I\subseteq J $, such that $I$ is $\m$-primary. Observe that either $J$ is also $\m$-primary or is $R$ itself. If $I\subseteq J$ is an $\m$-primary pair of ideals, then there is an ascending chain of ideals $I\subseteq (I, u_1)\subseteq (I,u_1,u_2)\subseteq \cdots \subseteq (I,u_1,u_2,...,u_{\lambda(J/I)})=J$ where $u_{i+1}\in (I,u_1,...,u_i):\m$. We shall let $I_0=I$ and $I_i=(I,u_1,...,u_i)$ for $1\leq i\leq \lambda(J/I).$}
\end{Remark}

\begin{Lemma}\label{Inequalities} Let $(R,\m)$ be a local ring of characteristic $p$ and $M$ a finitely generated $R$-module. If $I\subseteq J$ is an $\m$-primary pair of ideals, then $\lambda(J^{[q]}M/I^{[q]}M)\leq \lambda(M/\m^{[q]}M)\lambda(J/I).$

\end{Lemma}

\begin{proof} Observe that $\lambda(J^{[q]}M/I^{[q]}M)=\sum_{i=1}^{\lambda(J/I)}\lambda(I_i^{[q]}M/I_{i-1}^{[q]}M)$, hence it is enough to show that if $I$ is $\m$-primary and $u\in (I:\m)$, then $\lambda((I,u)^{[q]}M/I^{[q]}M)\leq \lambda(M/\m^{[q]}M)$. Well, $(I,u)^{[q]}M/I^{[q]}M\simeq M/(I^{[q]}M:_{M}u^q)$. Since $u\in I:\m$ we have that $\m^{[q]}M\subseteq (I^{[q]}M:_M u^q) $, hence $\lambda(M/(I^{[q]}M:_{M}u^q))\leq \lambda(M/\m^{[q]}M).$

\end{proof}

\begin{Proposition}\label{Main Lemma} Let $R$ be an F-finite ring and M a finitely generated $R$-module. There exists a constant $C>0$ such that for all $P\in\Spec(R)$ and $q=p^e$, if $IR_P\subseteq JR_P$ is a $PR_P$-primary pair of ideals, then $$\lambda\left(\frac{J^{[q]}M_P}{I^{[q]}M_P}\right)\leq Cq^{\dim(M_P)}\lambda\left(\frac{JR_P}{IR_P}\right).$$
\end{Proposition}

\begin{proof} By Lemma \ref{Inequalities} we only need to find a constant $C$ such that for all $P\in \Spec(R)$ and all $q$, $\lambda(M_P/P^{[q]}M_P)\leq Cq^{\dim(M_P)}.$ If $M_0\subseteq M_1\subseteq \cdots \subseteq M_n=M$ is a prime filtration of $M$ with $M_i/M_{i-1}\simeq R/P_i$, then $\lambda(M_P/P^{[q]}M_P)\leq \sum_{i=1}^{n}\lambda(R_P/(P_i+P^{[q]})R_P)$. This reduces the Proposition to showing that if $R$ is an F-finite domain, then there is a constant $C$ such that for all $P\in \Spec(R)$, $\lambda(R_P/P^{[q]}R_P)\leq Cq^{\heigh(P)}.$

\medskip

Suppose that $R$ is an F-finite domain. Let $\mathcal{S}(R)$ be the finite set of primes given by Proposition \ref{Global version of Dutta's Lemma} for the $R$-module $R$, and suppose that $P\in \Spec(R)$. By Lemma \ref{Length Lemma}, $\lambda(R_P/P^{[q]}R_P)=\frac{1}{q^{\alpha(R_P)}}\lambda(R_P^{1/q}/PR_P^{1/q})$, so it is equivalent to show that $\lambda(R_P^{1/q}/PR_P^{1/q})\leq Cq^{\gamma(R)}$ for some $C$ that does not depend on $P$. From the short exact sequence, $$0\rightarrow R^{q^{\gamma(R)}}\rightarrow R^{1/q}\rightarrow R^{1/q}/R^{q^{\gamma(R)}}\rightarrow 0,$$ we have that,
\begin{eqnarray*}
\lambda(R^{1/q}_P/PR_P^{1/q})&\leq& \lambda(R^{q^{\gamma(R)}}_P/PR^{q^{\gamma(R)}}_P)+\lambda(R^{1/q}/(R^{q^{\gamma(R)}}_P+PR_P^{1/q}))\\\\
&=&q^{\gamma(R)}+\lambda(R^{1/q}/(R^{q^{\gamma(R)}}_P+PR_P^{1/q})).
\end{eqnarray*}
Therefore we only need to find a constant $C$, independent of $P$, such that\\ $\lambda(R^{1/q}/(R^{q^{\gamma(R)}}_P+PR_P^{1/q}))\leq Cq^{\gamma(R)}$.

\medskip

Before localizing at $P$ we can apply Proposition \ref{Global version of Dutta's Lemma} to know that there exists a filtration of $R^{q^{\gamma(R)}}\subseteq R^{1/q}$, say $R^{q^{\gamma(R)}}=N_0\subseteq N_1\subseteq \cdots \subseteq N_n=R^{1/q}$, such that $n\leq C'|\mathcal{S}(R)|q^{\gamma(R)}$, where $C'$ is completely independent of $P$, and each $N_i/N_{i-1}\simeq R/P_i$ for some $P_i\in \mathcal{S}(R)$. For convenience let $M_i=(N_i)_P$. Localizing at $P$ and adding $PR_P^{1/q}$ to each module $M_i$ gives a filtration of $R^{q^{\gamma(R)}}_P+PR_P^{1/q}\subseteq R_P^{1/q}$ whose factors are $(M_i+PR_P^{1/q})/(M_{i-1}+PR_P^{1/q})\simeq M_i/(M_{i-1}+(PR_P^{1/q}\cap M_i))$. Noticing that $PR_P^{1/q}\cap M_i\supseteq PM_i$ we get that 
\begin{eqnarray*}
\lambda\left((M_i+PR_P^{1/q})/(M_{i-1}+PR_P^{1/q})\right)&=&\lambda\left(M_i/(M_{i-1}+(PR_P^{1/q}\cap M_i))\right)\\ \\
&\leq&\lambda\left(M_i/(M_{i-1}+PM_i)\right)\\ \\
&=&\lambda(R_P/(P_iR_P+PR_P))\\ \\
&\leq&\lambda(R_P/PR_P)=1.
\end{eqnarray*}
It now follows that $\lambda(R^{1/q}_P/PR_P^{1/q})\leq (1+C'|\mathcal{S}(R)|)q^{\gamma(R)}$.

\end{proof}

\begin{Corollary}\label{Isomorphic at min prime} Let $R$ be an F-finite ring, $N, M$ two finitely generated $R$-modules which are isomorphic at minimal primes of $R$. Then there is a constant $C$ such that for all $P\in \Spec(R)$ and $q=p^e$, if $IR_P\subseteq JR_P$ is a $PR_P$-primary pair of ideals, then $$\left|\lambda\left(\frac{J^{[q]}M_P}{I^{[q]}M_P}\right)-\lambda\left(\frac{J^{[q]}N_P}{I^{[q]}N_P}\right)\right|\leq Cq^{\heigh(P)-1}\lambda\left(\frac{JR_P}{IR_P}\right).$$

\end{Corollary}

\begin{proof} Using the notation in Remark \ref{Remark} and applying the triangle inequality $$\left|\lambda\left(\frac{J^{[q]}M_P}{I^{[q]}M_P}\right)-\lambda\left(\frac{J^{[q]}N_P}{I^{[q]}N_P}\right)\right|\leq \sum_{i=1}^{\lambda(J/I)}\left|\lambda\left(\frac{I_i^{[q]}M_P}{I_{i-1}^{[q]}M_P}\right)-\lambda\left(\frac{I_i^{[q]}N_P}{I_{i-1}^{[q]}N_P}\right)\right|.$$ Thus we may reduce ourselves to the scenario that $J=(I,u)$ where $u\in (I: P)$. There are exact sequences $M\xrightarrow{\varphi} N\rightarrow T_1\rightarrow 0$ and $N\xrightarrow{\psi} M\rightarrow T_2\rightarrow 0$, for which $T_1, T_2$ are $0$ when localized at minimal primes of $R$. Observe that $\varphi(I^{[q]}M:_M u^q)\subseteq (I^{[q]}N:_N u^q)$ so that there is induced map $\frac{M_P}{(I^{[q]}M_P:_{M_P}u^q)}\rightarrow \frac{N_P}{(I^{[q]}N_P:_{N_P}u^q)}$ whose cokernel, say $(T_1')_P$ is naturally the homomorphic image of $(T_1)_P$. Thus we have the following commutative diagram.
$$\begin{tikzpicture}
\node(1){$M_P$};
\node(2)[node distance=2in, right of=1]{$N_P$};
\node(3)[node distance=1in, below of=1]{$\frac{M_P}{(I^{[q]}M_P:_{M_P}u^q)}$};
\node(4)[node distance=2in, right of=3]{$\frac{N_P}{(I^{[q]}N_P:_{N_P}u^q)}$};
\node(5)[node distance=1.5in, right of=2]{$(T_1)_P$};
\node(6)[node distance=1in, right of=5]{$0$};
\node(7)[node distance=1.5in, right of=4]{$(T_1')_P$};
\node(8)[node distance=1in, right of=7]{$0$};
\draw[->](1) to node [above]{$\varphi$}(2);
\draw[->](1) to node [left]{}(3);
\draw[->](3) to node [above]{}(4);
\draw[->](2) to node [right]{$\pi_1$}(4);
\draw[->](2) to node [right]{}(5);
\draw[->](5) to node [right]{}(6);
\draw[->](5) to node [right]{$\pi_2$}(7);
\draw[->](4) to node [right]{}(7);
\draw[->](7) to node [right]{}(8);
\end{tikzpicture}$$
Therefore $\lambda\left(\frac{N_P}{(I^{[q]}N_P:_{N_P}u^q)}\right)-\lambda\left(\frac{M_P}{(I^{[q]}M_P:_{M_P}u^q)}\right)\leq \lambda((T_1')_P)$. Observe that $P^{[q]}N_P\subseteq (I^{[q]}N_P:_{N_P}u^q)$ so that $\pi_1(P^{[q]}N_P)=0$ and therefore $\pi_2(P^{[q]}(T_1)_P)=0$. Hence $(T_1')_P$ is the homomorphic image of $\frac{(T_1)_P}{P^{[q]}(T_1)_P}$. Thus
\begin{eqnarray*}
\lambda\left(\frac{(I,u)^{[q]}M_P}{I^{[q]}M_P}\right) &-& \lambda\left(\frac{(I,u)^{[q]}N_P}{I^{[q]}N_P}\right)\\ \\&=& \lambda\left(\frac{N_P}{(I^{[q]}N_P:_{N_P}u^q)}\right)-\lambda\left(\frac{M_P}{(I^{[q]}M_P:_{M_P}u^q)}\right)\leq \lambda\left(\frac{(T_1)_P}{P^{[q]}(T_1)_P}\right).
\end{eqnarray*}
A similar argument applied to the exact sequence $N\rightarrow M \rightarrow T_2 \rightarrow 0$ implies that $$\left|\lambda\left(\frac{J^{[q]}M_P}{I^{[q]}M_P}\right)-\lambda\left(\frac{J^{[q]}N_P}{I^{[q]}N_P}\right)\right|\leq \max_{i=1,2}\left\{\lambda\left(\frac{(T_i)_P}{P^{[q]}(T_i)_P}\right)\right\}.$$ The Corollary now follows by Proposition \ref{Main Lemma}.


\end{proof}

\begin{Corollary}\label{Corollary to Isomorphic at min prime} Let $R$ be an F-finite ring and $0\rightarrow M'\rightarrow M\rightarrow M''\rightarrow 0$ a short exact sequence of finitely generated $R$-modules. There exists a constant $C$ such that for all $P\in \Spec(R)$ and $q=p^e$, if $IR_P\subseteq JR_P$ is a $PR_P$-primary pair of ideals, then 
\begin{eqnarray*}
\left|\lambda\left(\frac{J^{[q]}M_P}{I^{[q]}M_P}\right)-\lambda\left(\frac{J^{[q]}M'_P}{I^{[q]}M'_P}\right)-\lambda\left(\frac{J^{[q]}M''_P}{I^{[q]}M''_P}\right)\right|\leq Cq^{\dim(M_P)-1}\lambda\left(\frac{JR_P}{IR_P}\right).
\end{eqnarray*}

\end{Corollary}

\begin{proof} Observe that $\lambda\left(\frac{(J+\tiny{\Ann}_R M)R_P}{(I+\tiny{\Ann}_R M)R_P}\right)\leq \lambda\left(\frac{JR_P}{IR_P}\right)$. Therefore we can begin by replacing $R$ with $R/\Ann_R M$ so that $\height(P)=\dim M_P$ for all $P\in \Spec(R)$. If $R$ is reduced then $M$ is isomorphic to $M'\bigoplus M''$ at minimal primes of $R$ and we can apply Corollary \ref{Isomorphic at min prime}. Suppose $R$ is not reduced. Using a standard argument, we can reduce to the scenario that $R$ is reduced. See for example the proofs of Lemma 1.5 in \cite{Monsky1983} and Proposition 3.11 in \cite{Huneke2013}. Let $e_0$ be a large enough integer so that for $q_0=p^{e_0}$, $\sqrt{0}^{[q_0]}=0$. Let $F: R\rightarrow R$ be the Frobenius endomorphism. Then $F^{e_0}(R)$ is abstractly isomorphic to the reduced ring $R/\sqrt{0}$ and $R$ is module finite over $F^{e_0}(R)$. Then for all $P\in \Spec(R)$ and $IR_P\subseteq PR_P$ which is $PR_P$-primary, $$\frac{1}{q_0^{\alpha(P)}}\lambda_{F^{e_0}(R_P)}\left(\frac{M_P}{(I^{[q_0]}\cap F^{e_0}(R))^{[q]}M_P}\right)=\lambda_{R_P}\left(\frac{M_P}{(I^{[q_0]}\cap F^{e_0}(R))^{[q]}M_P}\right)=\lambda_{R_P}\left(\frac{M_P}{I^{[qq_0]}M_P}\right).$$

\end{proof}

\begin{Theorem}\label{The main bound} Let $R$ be an F-finite ring and $M$ a finitely generated $R$-module. There exists a constant $C$ such that for all $P\in \Spec(R)$, for all $q_1,q_2$, if $IR_P\subseteq JR_P$ is a $PR_P$-primary pair of ideals, then $$\left|\lambda\left(\frac{J^{[q_1]}M_P}{I^{[q_1]}M_P}\right)q_2^{\heigh(P)}-\lambda\left(\frac{J^{[q_1q_2]}M_P}{I^{[q_1q_2]}M_P}\right)\right|\leq C q_2^{\dim(M_P)}q_1^{\dim(M_P)-1}\lambda\left(\frac{JR_P}{IR_P}\right).$$

\end{Theorem}

\begin{proof} As in the proof of Corollary \ref{Corollary to Isomorphic at min prime}, we may replace $R$ by $R/\Ann_R M$ so that $\dim(M_P)=\height P$ for all $P\in \Spec(R)$. If there is a short exact sequence $0\rightarrow M'\rightarrow M\rightarrow M''\rightarrow 0$, then 

\begin{eqnarray*}
\left|\lambda\left(\frac{J^{[q_1]}M_P}{I^{[q_1]}M_P}\right)q_2^{\heigh(P)}-\lambda\left(\frac{J^{[q_1q_2]}M_P}{I^{[q_1q_2]}M_P}\right)\right|&\leq& A_1+A_2+A_3+A_4.
\end{eqnarray*}
Where
\begin{eqnarray*}
A_1&=& \left|\lambda\left(\frac{J^{[q_1]}M_P}{I^{[q_1]}M_P}\right)-\lambda\left(\frac{J^{[q_1]}(M'_P\bigoplus M''_P)}{I^{[q_1]}(M'_P\bigoplus M''_P)}\right)\right|q_2^{\heigh(P)}\\ \\
A_2&=& \left|\lambda\left(\frac{J^{[q_1q_2]}M_P}{I^{[q_1q_2]}M_P}\right)-\lambda\left(\frac{J^{[q_1q_2]}(M'_P\bigoplus M''_P)}{I^{[q_1q_2]}(M'_P\bigoplus M''_P)}\right)\right|\\ \\
A_3&=&\left|\lambda\left(\frac{J^{[q_1]}M'_P}{I^{[q_1]}M'_P}\right)q_2^{\heigh(P)}-\lambda\left(\frac{J^{[q_1q_2]}M'_P}{I^{[q_1q_2]}M'_P}\right)\right|\\ \\
A_4&=&\left|\lambda\left(\frac{J^{[q_1]}M''_P}{I^{[q_1]}M''_P}\right)q_2^{\heigh(P)}-\lambda\left(\frac{J^{[q_1q_2]}M''_P}{I^{[q_1q_2]}M''_P}\right)\right|.
\end{eqnarray*}
By Corollary \ref{Corollary to Isomorphic at min prime} there is a constant $C$ such that $A_1\leq Cq_1^{\heigh(P)-1}\lambda\left(\frac{JR_P}{IR_P}\right)q_2^{\heigh(P)}$ and $A_2\leq C(q_2q_1)^{\heigh(P)-1}\lambda\left(\frac{JR_P}{IR_P}\right)$. Therefore by considering a prime filtration of the module $M$, we can reduce proving the theorem to the scenario that $M=R/P$ for some prime $P\in \Spec(R)$, i.e., we may assume that $M=R$ is an F-finite domain. Observe that by Lemma \ref{Length Lemma}, $\lambda\left(\frac{J^{[q_1q_2]}R_P}{I^{[q_1q_2]}R_P}\right)=\frac{1}{q_2^{\alpha(P)}}\lambda\left(\frac{J^{[q_1]}R^{1/q_2}_P}{I^{[q_1]}R^{1/q_2}_P}\right).$ Therefore the theorem is now reduced to showing that there is a constant $C$ independent of $P,I,J,q_1,q_2$ such that $$\left|\lambda\left(\frac{J^{[q_1]}R_P}{I^{[q_1]}R_P}\right)q_2^{\gamma(R)}-\lambda\left(\frac{J^{[q_1]}R^{1/q_2}_P}{I^{[q_1]}R^{1/q_2}_P}\right)\right|\leq Cq_2^{\gamma(R)}q_1^{\heigh(P)-1}\lambda\left(\frac{JR_P}{IR_P}\right).$$
As in the proof of Corollary \ref{Isomorphic at min prime} we can further reduce to the scenario that $J=(I,u)$ where $u\in (I:P)$.

\medskip

Let $C, \mathcal{S}(R)$ be as in Lemma \ref{Global version of Dutta's Lemma part 3} with corresponding inclusions of $R$-modules $R^{1/q}\rightarrow R^{q^{\gamma(R)}}$ and $R^{q^{\gamma(R)}}\rightarrow R^{1/q}$ whose cokernels are $T_1(q)$ and $T_2(q)$ respectively. So there are exact sequences $0\rightarrow R^{
1/q}\rightarrow R^{q^{\gamma(R)}}\rightarrow T_1(q)\rightarrow 0$ and $0\rightarrow R^{
q^{\gamma(R)}}\rightarrow R^{1/q}\rightarrow T_2(q)\rightarrow 0$ so that both $T_1(q)$ and $T_2(q)$ have a prime filtration whose prime factors are isomorphic to $R/Q$ where $Q\in \mathcal{S}(R)$ and such a prime factor appears no more than $Cq^{\gamma(R)}$ times in the filtration. As in the proof of Corollary \ref{Isomorphic at min prime} there will be the following commutative diagrams with all vertical maps being surjective.
$$\begin{tikzpicture}
\node(1){$R_P^{1/q_2}$};
\node(2)[node distance=2in, right of=1]{$R^{q_2^{\gamma(R)}}_P$};
\node(3)[node distance=1in, below of=1]{$\frac{R^{1/q_2}_P}{(I^{[q]}R^{1/q_2}_P:_{R^{1/q_2}_P}u^q)}$};
\node(4)[node distance=2in, right of=3]{$\frac{R^{q^{\gamma(R)}}_P}{(I^{[q]}:_{R}u^q)R_P^{\q^{\gamma(R)}}}$};
\node(5)[node distance=1.5in, right of=2]{$T_1(q_2)_P$};
\node(6)[node distance=1in, right of=5]{$0$};
\node(7)[node distance=1.5in, right of=4]{$T_1'(q_2)_P$};
\node(8)[node distance=1in, right of=7]{$0$};
\draw[->](1) to node [above]{}(2);
\draw[->](1) to node [left]{}(3);
\draw[->](3) to node [above]{}(4);
\draw[->](2) to node [right]{}(4);
\draw[->](2) to node [right]{}(5);
\draw[->](5) to node [right]{}(6);
\draw[->](5) to node [right]{}(7);
\draw[->](4) to node [right]{}(7);
\draw[->](7) to node [right]{}(8);
\end{tikzpicture}$$
$$\begin{tikzpicture}
\node(1){$R^{q_2^{\gamma(R)}}_P$};
\node(2)[node distance=2in, right of=1]{$R_P^{1/q_2}$};
\node(3)[node distance=1in, below of=1]{$\frac{R^{q^{\gamma(R)}}_P}{(I^{[q]}:_{R}u^q)R_P^{\q^{\gamma(R)}}}$};
\node(4)[node distance=2in, right of=3]{$\frac{R^{1/q_2}_P}{(I^{[q]}R^{1/q_2}_P:_{R^{1/q_2}_P}u^q)}$};
\node(5)[node distance=1.5in, right of=2]{$T_2(q_2)_P$};
\node(6)[node distance=1in, right of=5]{$0$};
\node(7)[node distance=1.5in, right of=4]{$T_2'(q_2)_P$};
\node(8)[node distance=1in, right of=7]{$0$};
\draw[->](1) to node [above]{}(2);
\draw[->](1) to node [left]{}(3);
\draw[->](3) to node [above]{}(4);
\draw[->](2) to node [right]{}(4);
\draw[->](2) to node [right]{}(5);
\draw[->](5) to node [right]{}(6);
\draw[->](5) to node [right]{}(7);
\draw[->](4) to node [right]{}(7);
\draw[->](7) to node [right]{}(8);
\end{tikzpicture}$$
Furthermore, $T_i'(q_2)_P$ will be the homomorphic image of $\frac{T_i(q_2)_P}{P^{[q]}T_i(q_2)_P}$ for $i=1,2$. It follows that 
\begin{eqnarray*}
\left|\lambda\left(\frac{J^{[q_1]}R_P}{I^{[q_1]}R_P}\right)q_2^{\gamma(R)}-\lambda\left(\frac{J^{[q_1]}R^{1/q_2}_P}{I^{[q_1]}R^{1/q_2}_P}\right)\right|\leq \max_{i=1,2}\left\{\lambda\left(\frac{T_i(q_2)_P}{P^{[q]}T_i(q_2)_P}\right)\right\}.
\end{eqnarray*}
For each $i=1,2$, $\lambda\left(\frac{T_i(q_2)_P}{P^{[q]}T_i(q_2)_P}\right)\leq Cq^{\gamma(R)}\max_{Q\in\mathcal{S}(R)}\{\left(\frac{R_P}{(Q+P^{[q]})R_P}\right)\}.$ We can now apply Proposition \ref{Main Lemma} to know that the desired bound exists.



\end{proof}

\section{Uniform Bounds in Rings Essentially of Finite Type Over an Excellent Local Ring}\label{Uniform Bounds 2}

The purpose of this section is to establish Proposition \ref{Main Lemma} and Theorem \ref{The main bound} for rings which are essentially of finite type over an excellent local ring. The following well known Lemma shall allow us to reduce our considerations to rings which are essentially of finite type over a complete local ring.

\begin{Lemma}\label{Regular fibers Lemma} Let $R\rightarrow S$ be a faithfully flat homomorphism of characteristic $p$ Noetherian rings with regular fibers. Let $M$ be a finitely generated $R$-module, $P\in \Spec(R)$ and $IR_P$ an $PR_P$-primary ideal. Then $$\frac{1}{q^{\heigh(P)}}\lambda_{R_P}\left(\frac{M_P}{I^{[q]}M_P}\right)=\frac{1}{q^{\heigh(Q)}}\lambda_{S_Q}\left(\frac{(S\otimes_RM)_Q}{(I,\underline{x})^{[q]}(S\otimes_RM)_Q}\right),$$ where $Q$ is a prime of $S$ lying over $P$ and $\underline{x}$ is a regular system of parameters for $S_Q/PS_Q.$

\end{Lemma}

\begin{proof} The first thing to observe is that $\frac{(S\otimes_R M)_Q}{(I,\underline{x})^{[q]}(S\otimes_R M)_Q}\simeq \frac{S_Q}{\underline{x}^{[q]}S_Q}\otimes_{R_P}\frac{M_P}{I^{[q]}M_P}$. Since $R_P\rightarrow S_Q$ is flat and $S_Q/PS_Q$ is regular, we have that
\begin{align*}
\lambda_{S_Q}\left(\frac{(S\otimes_R M)_Q}{(I,\underline{x})^{[q]}(S\otimes_R M)_Q}\right)=&\lambda_{S_Q}\left(\frac{S_Q}{\underline{x}^{[q]}S_Q}\otimes_{R_P}\frac{M_P}{I^{[q]}M_P}\right)\\\\
=&\lambda_{S_{Q}}\left(\frac{S_Q}{(P+\underline{x}^{[q]})S_Q}\right)\lambda_{R_{P}}\left(\frac{M_P}{I^{[q]}M_P}\right)\\\\
=& q^{\heigh(Q)-\heigh(P)}\lambda_{R_{P}}\left(\frac{M_P}{I^{[q]}M_P}\right).
\end{align*}
Dividing both sides of the equation by $q^{\heigh(Q)}$ gives the desired result.

\end{proof}

Suppose that $R$ is essentially of finite type over the excellent local ring $A$. Let $\hat{A}$ denote the completion of $A$ with respect to its maximal ideal. Then $R\rightarrow \hat{A}\otimes_A R$ is a faithfully flat homomorphism with regular fibers (\cite{Matsumura1980}, Section 33, Lemma 4). This observation and Lemma \ref{Regular fibers Lemma} allow us to reduce proving statements about rings essentially of finite type over an excellent local ring to rings which are essentially of finite type over a complete local ring.

\medskip

If $R$ is essentially of finite type over a complete local ring $A$, then let $\Lambda$ be a p-base of the residue field of $A$. We shall let $\Gamma$ be a cofinite subset of $\Lambda$. For each such $\Gamma$ there is an associated $R$-algebra, $R^{\Gamma}$, which satisfies the following.

\begin{Theorem}[\emph{\cite{HoHu1994}}, Section 6]\label{Gamma Construction} Let $R$ be a characteristic $p$ ring essentially of finite type over a complete local ring. Then for each $\Gamma\leq \Lambda$, $R^{\Gamma}$ is a faithfully flat, purely inseparable, F-finite $R$-algebra. 

\end{Theorem} 

To say that $R\rightarrow R^{\Gamma}$ is purely inseparable is to say that for each $s\in R^{\Gamma}$, there exists an $n\in\N$ such that $s^{n}\in R$. From this it follows that the induced map $\Spec(R^{\Gamma})\rightarrow \Spec(R)$ is a homeomorphism. The inverse map sends a prime $P\in \Spec(R)$ to $\sqrt{PR^{\Gamma}}$. If $P\in\Spec(R)$ we shall let $P_\Gamma=\sqrt{PR^{\Gamma}}$. 

\medskip

If $R$ is essentially of finite type over a complete local ring, then for each $\Gamma$ we have that $PR^{\Gamma}_{P_{\Gamma}}$ is $P_\Gamma R^{\Gamma}_{P_{\Gamma}}$-primary. If $R$ is essentially of finite type over an excellent local ring $A$, then $\Gamma$ shall represent a cofinite subset of a p-base for a coefficient field of $\hat{A}$. If $R$ is essentially of finite type over a complete local ring and $M$ a finitely generated $R$-module, then we let $M^{\Gamma}=R^{\Gamma}\otimes_R M$.

\begin{Proposition}\label{Main Lemma Finite Type} Let $R$ be essentially of finite type over an excellent local ring and let M be a finitely generated $R$-module. There exists a constant $C>0$ such that for all $P\in\Spec(R)$ and $q=p^e$, if $IR_P\subseteq JR_P$ is a $PR_P$-primary pair of ideals, then $$\lambda\left(\frac{J^{[q]}M_P}{I^{[q]}M_P}\right)\leq Cq^{\dim(M_P)}\lambda\left(\frac{JR_P}{IR_P}\right).$$
\end{Proposition}

\begin{proof} By Lemma \ref{Regular fibers Lemma} and the remarks that follow, we may reduce to the scenario that $R$ is essentially of finite type over a complete local ring. Choose any $\Gamma$. Then for each $P\in \Spec(R)$ one sees that by tensoring a prime filtration of $\frac{J^{[q]}M_P}{I^{[q]}M_P}$ with $R^{\Gamma}_{P_{\Gamma}}$ that $$\lambda_{R^{\Gamma}_{P_{\Gamma}}}\left(\frac{J^{[q]}M^{\Gamma}_{P_\Gamma}}{I^{[q]}M^{\Gamma}_{P_\Gamma}}\right)=\lambda_{R_P}\left(\frac{J^{[q]}M_P}{I^{[q]}M_P}\right)\lambda_{R^{\Gamma}_{P_{\Gamma}}}(R^{\Gamma}_{P_{\Gamma}}/PR^{\Gamma}_{P_{\Gamma}}).$$
We can now apply Proposition \ref{Main Lemma} to the F-finite ring $R^{\Gamma}$ so that we know there exists a constant $C$ such that for all $P\in \Spec(R)$ and for all $q$,
\begin{eqnarray*}
\lambda_{R_P}\left(\frac{J^{[q]}M_P}{I^{[q]}M_P}\right)&=&\frac{\lambda_{R^{\Gamma}_{P_{\Gamma}}}\left(J^{[q]}M^{\Gamma}_{P_{\Gamma}}/I^{[q]}M^{\Gamma}_{P_{\Gamma}}\right)}{\lambda_{R^{\Gamma}_{P_{\Gamma}}}(R^{\Gamma}_{P_{\Gamma}}/PR^{\Gamma}_{P_{\Gamma}})}\\ \\
&\leq&\frac{Cq^{\heigh(P_{\Gamma})}\lambda_{R^{\Gamma}_{P_{\Gamma}}}\left(JR^{\Gamma}_{P_{\Gamma}}/IR^{\Gamma}_{P_{\Gamma}}\right)}{\lambda_{R^{\Gamma}_{P_{\Gamma}}}(R^{\Gamma}_{P_{\Gamma}}/PR^{\Gamma}_{P_{\Gamma}})}\\ \\
&=&Cq^{\heigh(P)}\lambda_{R_P}\left(\frac{JR_P}{IR_P}\right).
\end{eqnarray*}

\end{proof}

\begin{Theorem}\label{The main bound finite type} Let $R$ be essentially of finite type over an excellent local ring and let $M$ be a finitely generated $R$-module. There exists a constant $C$ such that for all $P\in \Spec(R)$, for all $q_1,q_2$, if $IR_P\subseteq JR_P$ is a $PR_P$-primary pair of ideals, then $$\left|\lambda\left(\frac{J^{[q_1]}M_P}{I^{[q_1]}M_P}\right)q_2^{\heigh(P)}-\lambda\left(\frac{J^{[q_1q_2]}M_P}{I^{[q_1q_2]}M_P}\right)\right|\leq C q_2^{\heigh(P)}q_1^{\heigh(P)-1}\lambda\left(\frac{JR_P}{IR_P}\right).$$
\end{Theorem}

\begin{proof} The proof of this Theorem is identical to the proof of Proposition \ref{Main Lemma Finite Type}. Lemma \ref{Regular fibers Lemma} allows us to reduce to the scenario that $R$ is essentially of finite type over an excellent local ring. Pick a $\Gamma$ and let $C$ be as in Theorem \ref{The main bound} for the F-finite ring $R^\Gamma$, then 
\begin{align*}
&\left|\lambda_{R_P}\left(\frac{J^{[q_1]}M_P}{I^{[q_1]}M_P}\right)q_2^{\heigh(P)}-\lambda_{R_P}\left(\frac{J^{[q_1q_2]}M_P}{I^{[q_1q_2]}M_P}\right)\right|\\\\
&=\left. \left| \lambda_{R_{P_{\Gamma}}^{\Gamma}}\left(\frac{J^{[q_1]}M^{\Gamma}_{P_{\Gamma}}}{I^{[q_1]}(M^{\Gamma}_{P_{\Gamma}})}\right)q_2^{\heigh(P)}-\lambda_{R_{P_{\Gamma}}^{\Gamma}}\left(\frac{J^{[q_1q_2]}M^{\Gamma}_{P_{\Gamma}}}{I^{[q_1q_2]}M^{\Gamma}_{P_{\Gamma}}}\right)\right|\middle/  \lambda_{R_{P_{\Gamma}}^{\Gamma}}\left(\frac{R_{P_{\Gamma}}^{\Gamma}}{PR_{P_{\Gamma}}^{\Gamma}}\right)\right. \\\\ 
&\leq \frac{C q_2^{\heigh(P_\Gamma)}q_1^{\heigh(P_\Gamma)-1}\lambda_{R_{P_\Gamma}}\left(\frac{JR_{P_\Gamma}}{IR_{P_\Gamma}}\right)}{\lambda_{R_{P_{\Gamma}}^{\Gamma}}\left(\frac{R_{P_{\Gamma}}^{\Gamma}}{PR_{P_{\Gamma}}^{\Gamma}}\right)}= C q_2^{\heigh(P)}q_1^{\heigh(P)-1}\lambda_{R_P}\left(\frac{JR_P}{IR_P}\right).
\end{align*}

\end{proof}

\section{Uniform Convergence and Continuity Results}\label{Applications}

\begin{Theorem}\label{Main Theorem} Let $R$ be either F-finite or essentially of finite type over an excellent local ring and let $M$ be a finitely generated $R$-module. Let $I(-)$ be a map of primary ideals. The sequence of functions $\frac{\lambda_{q_1}(I(-))}{\lambda_1(I(-))}:\Supp(M)\rightarrow \R$, which sends a prime $P\in \Supp(M)$ to $\frac{\lambda(M_P/I(P)^{[q_1]}M_P)}{q_1^{\dim(M_{P}})\lambda(R_P/I(P)R_P)}$, converges uniformly to the scaled Hilbert-Kunz multiplicity function $\frac{\e(I(-), M_{-})}{\lambda_1(I(-))}$, which sends a prime $P\in\Supp(M)$ to $\frac{\e(I(P),M_{P})}{\lambda(R_P/I(P)R_P)}$ as $q_1\rightarrow \infty$.

\end{Theorem}

\begin{proof} Given $\epsilon >0$, our goal is to show that there exists a $q'$ such that for all $P\in \Supp(M)$ and for all $q_1\geq q'$, $|\frac{1}{\lambda(R_P/I(P)R_P)}\lambda_{q_1}(I(P))-\frac{1}{\lambda(R_P/I(P)R_P)}\e(I(P),M_P)|<\epsilon$. After modding out $R$ by $\Ann_R M$, it follows by Theorems \ref{The main bound} and \ref{The main bound finite type} that there exists a constant $C>0$ such that, for all $P\in \Supp(M)$ and for all $q_1,q_2$, 
\begin{align*}
&\left|\lambda\left(\frac{M_P}{I(P)^{[q_1]}M_P}\right)q_2^{\dim(M_P)}-\lambda\left(\frac{M_P}{I(P)^{[q_1q_2]}M_P}\right)\right|\\\\
&\leq Cq_2^{\dim(M_P)}q_1^{\dim(M_P)-1}\lambda\left(\frac{R_P}{(I(P)+\Ann_R(M))R_P}\right)\\\\
&\leq Cq_2^{\dim(M_P)}q_1^{\dim(M_P)-1}\lambda\left(\frac{R_P}{I(P)R_P}\right).
\end{align*}

\medskip

Dividing both sides of the inequality by $q_2^{\dim(M_P)}$, letting $q_2\rightarrow \infty$, and applying Lemma \ref{Multiplicity Lemma}, gives that for all $P\in \Supp(M)$ and for all $q_1$, $$\left|\lambda\left(\frac{M_P}{I(P)^{[q_1]}M_P}\right)-q_1^{\dim(M_P)}\e(I(P), M_P)\right|\leq Cq_1^{\dim(M_P)-1}\lambda\left(\frac{R_P}{I(P)R_P}\right).$$ 

\medskip

Choose $q'$ large enough that $\frac{C}{q'}<\epsilon$ and let $q_1\geq q'$. Dividing the above inequality by $q_1^{\dim(M_P)}\lambda(R_P/I(P)R_P)$ gives that for all $P\in\Supp(M)$ and all $q_1$, $$\left|\frac{\lambda^M_{q_1}(I(P))}{\lambda(R_P/I(P)R_P)}-\frac{\e(I(P), M_P)}{\lambda(R_P/I(P)R_P)}\right|\leq \frac{C}{q_1}<\epsilon.$$

\end{proof}

Let $n\in \N$ and set $f_n(P)= \frac{1}{q_1^{\dim(M_{P})}}\lambda(M_P/I(P)^{[q_1]}M_P)$ where $q_1=p^n$ and let $f$ be the limit function $f(P)=\e(I(P), M_P)$. What Theorem \ref{Main Theorem} is saying is that there exists a strictly positive function $g:\Spec(R)\rightarrow \R$, namely $g(P)=\frac{1}{\lambda(R_P/I(P)R_P)}$, which does not depend on $n$, such that $gf_n$ converges uniformly to the function $gf$. If  there exists a $\delta>0$ such that for all $P\in \Spec(R)$ $g(P)\geq \delta$, then $f_n$ converges uniformly to $f$. To see this we can choose $n$ so large so that for all $|gf_n-gf|<\epsilon \delta$. Then $|f_n-f|<\epsilon\delta/g\leq \epsilon\delta/\delta=\epsilon$. Using this observation we obtain the following Corollary to Theorem \ref{Main Theorem}.

\begin{Corollary}\label{Corollary of Main Theorem} Let $R$ be an F-finite ring of prime characteristic $p>0$ and let $M$ be a finitely generated $R$-module. Let $I(-)$ be a map of primary ideals. Suppose that there exists a $q$ such that $P^{[q]}\subseteq I(P)$ for all $P\in\Supp(M)$, or more generally there exists a constant $D$ such that $\lambda(R_P/I(P)R_P)\leq D$ for all $P\in\Supp(M)$. Then the sequence of functions $\lambda_{q_1}(I(-)):\Supp(M)\rightarrow \R$, which sends a prime $P$ to $\frac{1}{q_1^{\dim(M_P)}}\lambda(M_P/I(P)^{[q_1]}M_P)$, converges uniformly to the Hilbert-Kunz multiplicity function $\e(I(-), M_{-})$, which sends a prime $P$ to $\e(I(P),M_{P})$.
\end{Corollary}

\begin{proof} By the above remarks we only need to find $\delta>0$ such that for all $P\in \Supp(M)$, $\frac{1}{\lambda(R_P/I(P)R_P)}\geq \delta$, or equivalently that there exists a $D$ such that for all $P\in \Supp(M)$, $\lambda(R_P/I(P)R_P)\leq D$. We are assuming that for each $P\in\Spec(R)$ that $P^{[q]}\subseteq I(P)$. Hence by Lemma \ref{Main Lemma} there exists a constant $C$ such that for all $P\in\Supp(M)$, $\lambda(R_P/I(P)R_P)\leq \lambda(R_P/P^{[q]}R_P)\leq Cq^{\heigh(P)}\leq Cq^{\dim(R)}$. Therefore $D=Cq^{\dim(R)}$ works.

\end{proof}

Corollary \ref{Corollary of Main Theorem} gives an alternative proof of Smirnov's result that if $R$ is F-finite or essentially of finite type over an excellent local ring, then $\e(-)$ is upper semi-continuous at primes $P$ such that $R_P$ is equidimensional.

\begin{Corollary}\label{Corollary 1} Let $R$ be either F-finite or essentially of finite type over an excellent local ring, then the Hilbert-Kunz function $\e(-):\Spec(R)\rightarrow \R_{\geq 1}$ which sends a prime $P\mapsto \e(R_P)$ is upper semi-continuous at all $P\in \Spec(R)$ such that $R_P$ is equidimensional.

\end{Corollary}

\begin{proof} Consider the map of primary ideals $I(-)$ which sends a prime $P$ to $P$. Then $P^{[1]}=P\subseteq I(P)$ for each $P\in \Spec(R)$. Corollary \ref{Corollary of Main Theorem} says that $\lambda_{q_1}(-)$ converges uniformly to $\e(-)$. E. Kunz originally showed in \cite{Kunz1976} that for each $q_1$ the function $\lambda_{q_1}(-)$ which sends a prime $P\mapsto \frac{1}{q^{\heigh(P)}}\lambda(R_P/P^{[q]}R_P)$, is upper semi-continuous on all rings which are locally equidimensional. If $R_P$ is equidimensional, then $R$ being catenary implies that there is an $s\in R-P$ such that $R_s$ is locally equidimensional. The $s$ which works is $1$ if $\min(R_P)=\min (R)$. If $\min(R_P)\subsetneq \min(R)$, then just choose $s\in \cap_{Q\in Min(R)-\min(R_P)}Q\setminus P$. Therefore, if $R_P$ is equidimensional, then in an open neighborhood of $P$, $\e(-)$ is the uniform limit of upper semi-continuous functions, hence $\e(-)$ is upper semi-continuous as well.

\end{proof}

\begin{Lemma}\label{F-signature 6} Let $(R,\m,k)$ be an excellent reduced local ring of dimension $d$. Let $q_1,q_2$ equal $p^{e_1}$ and $p^{e_2}$ respectively and $b_{q_1}=q_1^ds_{q_1}, b_{q_1q_2}=(q_1q_2)^ds_{q_1q_2}$, where $s_{q_1}$ and $s_{q_1q_2}$ are the $q_1$th and $q_1q_2$th normalized Frobenius splitting numbers of $R$ respectively. Then there is an irreducible $\m$-primary ideal $I$ and $u\in (I: \m)$ such that $b_{q_1}=\lambda((I,u)^{[q_1]}/I^{[q_1]})$ and $b_{q_1q_2}= \lambda((I,u)^{[q_1q_2]}/I^{[q_1q_2]})$. 

\end{Lemma}

\begin{proof} Let $I_e=\{r\in R \mid \F r\otimes u=0\mbox{ in }\F R\otimes_R E_R(k)\}$ where $u$ generates the socle of $E_R(k)$. Then $b_q=\lambda(R/I_e)$, (\cite{Yao2006}, Remark 2.3). Since $R$ is reduced and excellent, $R$ is approximately Gorenstein (\cite{Hochster1977}, Theorem 1.7).  So there exists a descending chain of irreducible $\m$-primary ideals $\{I_t\}_{t\in \N}$ which is cofinal with $\{\m^t\}_{t\in \N}$. Let $u_t$ generate the socle mod $I_t$. Then $I_e=\cup_{t=1}^{\infty}(I_t^{[q]}:u_t^q)$, therefore for each $q$ there is a $t_0$ such that for all $t\geq t_0$, $b_q=\lambda(R/(I_t^{[q]}:u_t^q))=\lambda((I_t,u_t)^{[q]}/I_t^{[q]}).$

\end{proof}

\begin{Theorem}\label{F-signature 7} Let $R$ be either F-finite or essentially of finite type over an excellent local ring. There exists a constant $C$ such that, for all $P\in \Spec(R)$, and for all $q_1,q_2$, $$|b_{q_1}(P)q_2^{\heigh(P)}-b_{q_1q_2}(P)|\leq Cq_2^{\heigh(P)}q_1^{\heigh(P)-1}.$$
\end{Theorem}

\begin{proof} It is well known that if $b_q(P)>0$ for some, equivalently for all, $q$, then $R_P$ is a reduced ring. Therefore $C=0$ is a constant which works for all $P\in \Spec(R)$ such that $R_P$ is not reduced. If $R_P$ is reduced there exists an $s\in R-P$ such that $R_s$ is reduced. Therefore by quasi-compactness of $\Spec(R)$, we may reduce our considerations to when $R$ is a reduced ring. The Theorem now follows by Lemma \ref{F-signature 6}, Theorem \ref{The main bound}, and Theorem \ref{The main bound finite type}.

\end{proof}

\begin{Theorem}\label{F-signature 8} Let $R$ be either F-finite or essentially of finite type over an excellent local ring. The $q$th normalized Frobenius splitting number function, which maps a prime $P\mapsto s_q(P)$, converges uniformly to the F-signature function, which maps a prime $P\mapsto s(R_P)$ as $q\rightarrow \infty.$

\end{Theorem}

\begin{proof} Let $\epsilon >0$, let $C$ be as in Theorem \ref{F-signature 7}, and choose $q$ so large that $\frac{C}{q}<\epsilon.$ Then for all $P\in \Spec(R)$ we have that $$|b_{q_1}(P)q_2^{\heigh(P)}-b_{q_1q_2}(P)|\leq Cq_2^{\heigh(P)}q_1^{\heigh(P)-1}.$$ Therefore $$\left|b_{q_1}(P)-q_1^{\heigh(P)}\frac{b_{q_1q_2}(P)}{(q_1q_2)^{\heigh(P)}}\right|\leq Cq_1^{\heigh(P)-1}.$$ Letting $q_2\rightarrow \infty$ we have that for all $P\in \Spec(R)$ that $$|b_{q_1}(P)-q_1^{\heigh(P)}s(R_P)|\leq Cq_1^{\heigh(P)-1}.$$ Hence for all $q_1\geq q$ and all $P\in \Spec(R)$, $$\left|\frac{b_{q_1}(P)}{q_1^{\heigh(P)}}-s(R_P)\right|\leq \frac{C}{q_1}\leq \frac{C}{q}<\epsilon.$$ This verifies that $\frac{b_{q}(P)}{q^{\heigh(P)}}=s_q(P)$ converges uniformly to $s(R_P)$ as $q\rightarrow \infty.$

\end{proof}

\medskip

\begin{Theorem}\label{F-signature 9} Let $R$ be either F-finite or essentially of finite type over an excellent local ring. The F-signature function on $\Spec(R)$ is lower semi-continuous. 
\end{Theorem}

\begin{proof} Let $\epsilon>0$ and let $P\in \Spec(R)$. If $s(R_P)=0$ then it is the case that for all $Q\in\Spec(R)$, that $s(R_P)-s(R_Q)\leq 0<\epsilon.$ Now suppose that $s(R_P)>0$. Aberbach and Leuschke showed in \cite{AberbachLeuschke2003}, along with Tucker's proof of the existence of $s(R_P)$ in \cite{Tucker2012}, that $s(R_P)>0$ if and only if $R_P$ is strongly F-regular. In particular we have that $R_P$ is a domain. There then exists an $s\in R-P$ such that $R_s$ is a domain. Enescu and Yao showed in \cite{EnescuYao2011} that if $S$ is a locally equidimensional ring which is either F-finite or essentially of finite type over an excellent local ring, then the $q$th normalized Frobenius splitting number function is lower semi-continuous on $\Spec(S)$. By Theorem \ref{F-signature 6}, we have that in a neighborhood of $P$, the F-signature function is the uniform limit of lower semi-continuous functions, hence itself is lower semi-continuous at $P$.

\end{proof}

Observe that Theorem \ref{F-signature 8} applied to the maximal ideal of an excellent local ring $(R,\m, k)$ directly shows the sequence $\frac{b_q}{q^{\dim(R)}}$ is a Cauchy sequence.

\section{Lower Semi-Continuity of F-signature of Pairs}\label{Lower Semi-Continuity of F-signature of Pairs}

In this section all rings under consideration will be F-finite. We want to establish the lower semi-continuity of the F-signature of a pair $(R,\mathscr{D})$ where $\mathscr{D}$ is a Cartier algebra, see section 2 of \cite{BST2012} for a more in-depth look at the basic notions of a Cartier subalgebra. Our main tool will be Proposition \ref{Main Lemma} in order to establish a uniform convergence result and the desired lower semi-continuity.

\medskip

Let $\mathscr{C}_q:= \Hom_R(F^q_*R,R)$ and $\mathscr{C}^R=\bigoplus_{q=p^e, e\geq 0} \mathscr{C}_q$. If $\varphi\in \CC_q$ and $\psi\in \CC_{q'}$ then $\varphi\cdot\psi:= \varphi\circ F^q_*\psi\in \CC_{qq'}$ where $F^q_*\psi(F^{qq'}_*r):= F^q_* \psi(F^{q'}_*r)$. We call $\mathscr{C}^R$ the \emph{(total) Cartier algebra} of $R$. Note that with the multiplication defined on homogenous elements of $\CC^R$ makes $\CC^R$ a noncommutative $\mathbb{F}_p$-algebra. Even though the 0th graded piece of $\CC^R$ is $\CC_{p^0}=\CC_1=\Hom_R(R,R)\simeq R$, $R$ is not central in $\CC^R$, hence $\CC^R$ is not an $R$-algebra. We say that $\DD\subseteq \CC^R$ is a \emph{Cartier subalgebra} of $R$ if $\DD$ is a $\mathbb{F}_p$-subalgebra of $\CC^R$ and $\DD_1=\CC_1\simeq R$.

\medskip

Suppose that $(R,\m,k)$ is a local ring and $\DD$ a Cartier subalgebra of $R$. Suppose that $F^q_*R\simeq \bigoplus M_i$ as an $R$-module. The summand $M_i$ is called a \emph{$\DD$-summand} if $M_i\simeq R$ and the projection $F^q_*R\rightarrow M_j\simeq R$ is an element of $\DD_q$. The \emph{qth F-splitting number of $(R,\DD)$} is the maximal number $a_q^\DD$ of $\DD$-summands appearing in the various direct sum decompositions of $F^q_*R$. Observe that $a_q^\CC=a_q$ for all $q$, the usual $q$th F-splitting number of $R$. For each $q=p^e$ let $I_q^\DD=\{r\in R\mid \varphi(F^q_*r)\in\m\mbox{ for all }\varphi\in \DD_q\}.$ The following Lemma is a list of basic properties about the sets $I_q^{\DD}$ which can all be found in Section 3 of \cite{BST2012}.

\medskip


\begin{Lemma}\label{BSTLemma} Let $(R,\m,k)$ be a local F-finite ring and $\DD$ a Cartier subalgebra and let $q,q_1,q_2$ be various powers of $p$ and $\varphi\in \DD_{q_1}$. Then 
\begin{enumerate}
\item $I_q^\DD\subseteq R$ is an ideal,
\item $\m^{[q]}\subseteq I_q^\DD$,
\item $\varphi(F^{q_1}_*I^\DD_{q_1q_2})\subseteq I^\DD_{q_2}$,
\item $\lambda(R/I_q^\DD)=\frac{a_q^\DD}{q^{\alpha(R)}}$.
\end{enumerate}

\end{Lemma}

Let $(R,\m, k)$ be local and $\DD$ a Cartier subalgebra. Set $\Gamma_\DD$ to be the semigroup $\{q\mid a_q^{\DD}\not=0\}$. The main result of Blickle, Schwede, and Tucker in \cite{BST2012} is that if $(R,\m, k)$ is local and $\DD$ a Cartier subalgebra, then the limit $\lim_{q\in \Gamma_{\DD}\rightarrow \infty}\frac{a_q^{\DD}}{q^{\alpha(R)+\dim(R)}}=\lim_{q\in \Gamma_\DD\rightarrow \infty} \frac{1}{q^{\dim(R)}}\lambda(R/I_q^{\DD})$ exists, it is called the \emph{F-signature of the pair $(R,\DD)$}, and is denoted $s(R,\DD)$. 

\medskip

Suppose that $R$ is F-finite and not necessarily local. Let $\DD$ be a Cartier algebra of $R$. Suppose that $S\subseteq R$ is a multiplicatively closed set. Since $R$ is F-finite, $S^{-1}\Hom_R(F^q_*R,R)\simeq \Hom_{S^{-1}R}(F^q_*S^{-1}R, S^{-1}R)$. Therefore there is a naturally induced Cartier subalgebra $S^{-1}\DD$ of $S^{-1}R$ such that $(S^{-1}\DD)_q=S^{-1}(\DD_q)$. If $P\in \Spec(R)$ and $s\in R$ we write $\DD_P$ and $\DD_s$ for the induced Cartier subalgebra of $R_P$ and $R_s$ respectively. For each $P\in\Spec(R)$ let $a_q(P,\DD)$ be the $q$th F-splitting number of $(R_P,\DD_P)$, $s_q(P,\DD)=a_q(P,\DD)/q^{\alpha(P)+\tiny{\height}(P)}$, and let $s(P,\DD)=s(R_P,\DD_P)$. Then $s_q(P,\DD):\Spec(R)\rightarrow \R$ converges to $s(P,\DD): \Spec(R)\rightarrow \R$ as $q\in \Gamma_{\DD}\rightarrow \infty $ as a limit of functions. For each $q$ and $P\in \Spec(R)$, let $$I^{\DD}_q(P)=\{r\in R_P\mid \forall \varphi \in (\DD_P)_q, \varphi(F^{q}_*r)\in PR_P\}.$$

\medskip

\begin{Remark}\label{Natural observation} If $R$ is an F-finite ring, $\DD$ a Cartier subalgebra, $N,M\in \N$, $f\in \DD_{q_1}^{N}$, $g\in \DD_{q_2}^{M}$, then the natural map $\underbrace{(g,g,...,g)}_\text{N times}\circ F^{q_2}_* f: F^{q_1q_2}_* R\rightarrow R^{NM}$ is an element of $\DD_{q_1q_2}^{NM}$.

\end{Remark}
\noindent The remark follows from the assumption that if $\varphi \in \DD_{q_1}$ and $\psi\in \DD_{q_2}$, then $\psi\circ F^{q_2}_* \varphi \in \DD_{q_1q_2}$.

\medskip

Our first goal will to be establish a version of Corollary \ref{Global version of Dutta's Lemma part 2} for a pair $(R,\DD)$ when $R$ is an F-finite domain. If $R$ is a not necessarily local F-finite domain and $\DD$ a Cartier subalgebra, we let $\Gamma_\DD=\Gamma_{\DD_0}$. We will now only be interested in containments of $R$-modules $R^{1/q}\subseteq R^{q^{\gamma(R)}}$ when $q\in \Gamma_\DD$. Not only will we need to know that for each $q\in \Gamma_\DD$ that a prime filtration of $R^{1/q}\subseteq R^{q^{\gamma(R)}}$ has only finitely many prime factors up to isomorphism and such prime factors only appears a controlled number of times, but we will need that each of the following $q^{\gamma(R)}$ maps is an element of $\DD_q$, $R^{1/q}\subseteq R^{q^{\gamma(R)}}\xrightarrow{\pi_i} R$,  where $\pi_i$ is the projection onto the $i$th factor.

\begin{Lemma}\label{Dutta Lemma for pairs} Let $R$ be an F-finite domain and $\DD$ a Cartier subalgebra. There exists a finite set of nonzero primes $\mathcal{S}(R, \DD)$ and a constant $C$ such that for every $q\in \Gamma_\DD$,

\begin{enumerate}

\item there is a containment of $R$-modules $R^{1/q}\subseteq R^{q^{\gamma(R)}}$ which is an element of $\DD_q^{q^{\gamma(R)}}$,

\item which has a prime filtration whose prime factors are isomorphic to $R/P$, where $P\in\mathcal{S}(R, \DD)$,

\item and for each $P\in \mathcal{S}(R, \DD)$, the prime factor $R/P$ appears no more than $Cq^{\gamma(R)}$ times in the prime filtration of the containment $R^{1/q}\subseteq R^{q^{\gamma(R)}}$.

\end{enumerate}

\end{Lemma}

\begin{proof} Let $\Gamma_\DD$ be generated by $\Lambda_\DD=\{q_1,...,q_m\}$ as a semigroup. For each $q_i\in\Lambda_\DD$ we can fix an embedding $R^{1/q_i}\subseteq R^{q_i^{\gamma(R)}}$ which is an element of $\DD_{q_i}^{q_i^{\gamma(R)}}$ which is an isomorphism when localized at $0$. To see this, let $W=R-0$ and $q\in \Lambda_\DD$ so that $R^{1/q}_W\simeq R_W^{q^{\gamma(R)}}$.  As $R_W$ is a field, $\Hom_{R_W}(R^{1/q}_W, R_W)\simeq R_W^{1/q}$ as an $R_W^{1/q}$-module. Suppose that $0\not=\varphi\in \DD_q$, then $\varphi_W$ generates $\Hom_{R_W}(R^{1/q}_W, R_W)\simeq R_W^{1/q}$ as an $R_W^{1/q}$-module. As $\DD_0=\Hom_R(R,R)$, we have that a $R_W^{1/q}$-multiple of $\varphi_W$ is still an element of $(\DD_W)_q$.  Therefore the isomorphism $R_W^{1/q}\simeq R_W^{q^{\gamma(R)}}$ is an element of $(\DD_W)_q^{q^{\gamma(R)}}$. As $R$ is an F-finite domain, the isomorphism $R_W^{1/q}\simeq R_W^{q^{\gamma(R)}}$ is the localization of an embedding $R^{1/q}\subseteq R^{q^{\gamma(R)}}$ which is an element of $\DD_q^{q^{\gamma(R)}}$. 

\medskip

For each $q\in \Lambda_{\DD}$ we consider a prime filtration of $R^{1/q}\subseteq R^{q^{\gamma(R)}}$, say $R^{1/q}=N_0\subseteq N_1\subseteq \cdots \subseteq N_n= R^{q^{\gamma(R)}}$, say $N_i/N_{i-1}\simeq R/P_i$. Let $\mathcal{S}(R/P_{q,i})$ and $C_{q,i}$ be as in Lemma \ref{Global version of Dutta's Lemma} and let $\mathcal{S}_q(R,\DD)=\bigcup_{i=1}^n \mathcal{S}(R/P_{q,i})\cup \{P_{q,i}\}$ and $\mathcal{S}(R,\DD)=\bigcup_{q\in \Lambda_\DD}\mathcal{S}_q(R,\DD)$. We can now set $C'=\sum C_{q,i}$. Every $q\in \Gamma_\DD$ can be expressed as $\prod_{q_i\in \Lambda_\DD} q_i^{e_i}$ where $e_i\in \N$. We show by induction on $\sum e_i$ that for each $q\in \Gamma_\DD$ there is a containment of $R$-modules $R^{1/q}\subseteq R^{q^{\gamma(R)}}$ which is an element of $\DD_q^{q^{\gamma(R)}}$, which has a prime filtration whose prime factors are isomorphic to $R/P$, where $P\in \mathcal{S}(R, \DD)$, and such a prime factor appears no more than $C'q^{\gamma(R)}\left(1+\frac{1}{p}+\cdots +\frac{1}{q}\right)$ times in the filtration. This trivially holds for $\sum e_i=1$. 

\medskip

Now suppose that $q=\prod_{q_i\in \Lambda_\DD}q_i^{e_i}$ with $\sum e_i>1$. Without loss of generality we may suppose that $e_1\geq 1$ so that $q'=\frac{q}{q_1}\in \Gamma_\DD$. By induction, we can find $R^{1/q'}=N_0\subseteq N_1\subseteq \cdots \subseteq N_m=R^{q'^{\gamma(R)}}$ is a prime filtration of an embedding $R^{1/q'}\subseteq R^{q'^{\gamma(R)}}$ in $\DD_{q'}^{q'^{\gamma(R)}}$, each $N_j/N_{i-j}\simeq R/P_j$ for some $P_j\in \mathcal{S}(R,\DD)$, and such a prime factor appears no more than $C'q'^{\gamma(R)}(1+\frac{1}{p}+\frac{1}{p^2}+\cdots+\frac{1}{q'})$ times in the filtration. Therefore $(R^{q_1^{\gamma(R)}})^{1/q'}=(R^{1/q'})^{q_1^{\gamma(R)}}\subseteq (R^{q'^{\gamma(R)}})^{q_1^{\gamma(R)}}=R^{q^{\gamma(R)}}$ has a prime filtration with prime factors $R/P_{j}$ with $P_j\in \mathcal{S}(R,\DD)$ and such a prime factor appears no mare than $C'q^{\gamma(R)}(1+\frac{1}{p}+\frac{1}{p^2}+\cdots+\frac{1}{q})$ times in the filtration. Furthermore, the prime filtration $R^{1/q_1}=N_{q_1,0}\subseteq N_{q_1,1}\subseteq\cdots \subseteq N_{q_1,n}= R^{q_1^{\gamma(R)}}$ gives the following filtration of $R^{1/q}=(R^{1/q_1})^{1/q'}\subseteq (R^{q_1^{\gamma(R)}})^{1/q'}$, $$(R^{1/q_1})^{1/q'}=N_{q_1,0}^{1/q'}\subseteq N_{q_1,1}^{1/q'}\subseteq \cdots \subseteq R_{q_1,n}^{1/q'}=(R^{q_1^{\gamma(R)}})^{1/q'}.$$ Since $N_{q_1,i}^{1/q'}/N_{q_1,i-1}^{1/q'}\simeq (R/P_{q_1,i})^{1/q'}$, we apply Lemma \ref{Global version of Dutta's Lemma} to know there is a prime filtration of each $N^{1/q'}_{q_1,i-1}\subseteq N^{1/q'}_{q_1,i}$ whose prime factors come from $\mathcal{S}(R/P_{q_1,i})$ and such a prime factor appears no more than $C_iq^{\gamma(R/P_{q_1,i})}\leq C'q^{\gamma(R)-1}$ times in the filtration. Putting all of this information together we get an embedding $R^{1/q}\subseteq R^{q^{\gamma(R)}}$, which is an element of $\DD_q^{q^{\gamma(R)}}$ by Remark \ref{Natural observation}, with a prime filtration whose prime factors come from $\mathcal{S}(R,\DD)$, and such a prime factor appears no more than the following number in the filtration,  $$q_1^{\gamma(R)}C'q'^{\gamma(R)}\left(1+\frac{1}{p}+\cdots +\frac{1}{q'}\right)+\sum_{i=1}^hC_iq^{\gamma(R/P_{q_1,i})}\leq C'q^{\gamma(R)}(1+\frac{1}{p}+\cdots +\frac{1}{q'}+\frac{1}{q})\leq 2C'q^{\gamma(R)}.$$

\end{proof}

Enescu and Yao showed that $a_q(P,\mathscr{C}):\Spec(R)\rightarrow \R$, hence $s_q(P,\mathscr{C})$, is lower semi-continuous on an F-finite ring which is locally equidimensional, (Corollary 2.5, \cite{EnescuYao2011}). We provide a very similar proof that shows $a_q(P,\DD)$, hence $s_q(P,\DD)$, is a lower semicontinuous function for any Cartier subalgebra $\DD$ whenever $R$ is locally equidimensional. It is well known that a function $f:X\rightarrow \R$, $X$ a topological space, is lower semi-continuous if and only if for each $r\in \R$ the sets $f^{-1}((r,\infty))=\{x\in X\mid f(x)>r\}$ is open in $X$. 

\medskip

Lower semi-continuity is a local condition. We may assume $R_P$, hence $R$ is reduced, else $a_q(P,\DD)=0$ and $a_q(-,\DD)$ is trivially lower semi-continuous at $P$. Suppose that $q\in \Gamma_\DD$, $r\in \R$, and let $P\in \{Q\in \Spec(R)\mid a_q(Q,\DD)> r\}$. Then $R_P^{1/q}\simeq R^{a_q(P,\DD)}_P\bigoplus M_P$ is such that each of the $a_q(P,\DD)$ projections $R^{1/q}_P\rightarrow R_P$ is an element of $(\DD_P)_q$. It follows that there is an $s\in R-P$ such that $R_s^{1/q}\simeq R_s^{a_q(P,\DD)}\bigoplus M_s$ and each of the $a_q^\DD(P)$-projections $R_s^{1/q}\rightarrow R_s$ is an element of $(\DD_s)_q$. Hence for all $P'\in D(s)$, $a_q(P',\DD)\geq a_q(P,\DD))>r$ and $\{Q\in \Spec(R)\mid a_q^\DD(Q)>r\}$ is indeed an open set. This shows that $a_q^\DD(P)$ is a lower semi-continuous and so is $s_q(P, \DD)$ since $a_q(P, \DD)$ and $s_q(P, \DD)$ differ only by a constant on connected components of $\Spec(R)$.

\medskip

Consider the following condition we could impose on a Cartier subalgebra $\DD$. 
\begin{eqnarray}
\label{star}(I^{\DD}_q(P))^{[p]}\subseteq I_{qp}^{\DD}(P)
\end{eqnarray}
Suppose $R$ is an F-finite domain and $\DD$ a Cartier subalgebra. Then $r\in I^{\DD}_q(P)$ if and only if $\varphi(r^{1/q})\in PR_P$ for all $\varphi\in \DD_q\subseteq \Hom_R(R^{1/q},R)$. Thus to impose condition (\ref{star}) is to impose that for each $r\in I_q^{\DD}(P)$ that $\psi(r^{1/q})\in PR_P$ for all $\psi\in \DD_{qp}\subseteq \Hom_R(R^{1/qp},R)$. This condition is seen to be satisfied if for each $\psi\in \DD_{qp}$ we require $\varphi\circ i\in \DD_q$ where $i $ is the natural inclusion $R^{1/q}\subseteq R^{1/qp}$.



\begin{Theorem}\label{Uniform Convergence of Pairs} Let $R$ be an F-finite domain and $\DD$ a Cartier subalgebra of $R$. Then the F-signature function which sends $P\in \Spec(R)$ to $s(P, \DD)$ is lower semi-continuous. Moreover, if the the Cartier subalgebra satisfies \emph{(\ref{star})}, then the function $s_q(P, \DD)$ converges uniformly to the F-signature function $s(P,\DD)$ as $q\in\Gamma_\DD\rightarrow \infty$.

\end{Theorem}

\begin{proof}  Let $C, \mathcal{S}(R)$ be as in Lemma \ref{Dutta Lemma for pairs}.  Let $q_1\in \Gamma_\DD$ so that $\DD_{q_1}\not=0$. Let $S(q_1)$ be the cokernel of $R^{1/q_1}\rightarrow R^{q_1^{\gamma(R)}}$. Therefore we have the following short exact sequences $$0\rightarrow R^{1/q_1}\rightarrow R^{q_1^{\gamma(R)}}\rightarrow S(q_1)\rightarrow 0. $$ By Remark \ref{Natural observation} we have exact sequences $$\frac{R^{1/q_1}_P}{I^\DD_{q_1q_2}(P)^{1/q_1}}\rightarrow \frac{R^{q_1^{\gamma(R)}}_P}{I^\DD_{q_2}(P)R_P^{q_1^{\gamma(R)}}}\rightarrow \tilde{S}(q_1)\rightarrow 0,$$ where $\tilde{S}(q_1)$ is the homomorphic image of $S(q_1)_P/I^\DD_{q_1}(P)S(q_1)_P$. Therefore by parts (2) and (4) of Lemma \ref{BSTLemma}, $$\frac{a_{q_2}(P,\DD)}{q_2^{\alpha(P)}}q_1^{\gamma(R)}-\frac{a_{q_1q_2}(P,\DD)}{(q_1q_2)^{\alpha(P)}}q_1^{\alpha(P)}\leq \lambda(\tilde{S}(q_1)_P)\leq \lambda\left(\frac{S(q_1)_P}{I^\DD_{q_2}(P)S(q_1)_P}\right)\leq \lambda\left(\frac{S(q_1)_P}{P^{[q_2]}S(q_1)_P}\right) .$$ By Proposition \ref{Main Lemma} there is a constant $C_1$, independent of $P,q_2$, such that $$\max_{Q\in \mathcal{S}(R)}\lambda\left(R_P/(Q+P^{[q_2]})R_P\right)\leq C_1q_2^{\heigh(P)-1}.$$ It follows that $$\frac{a_{q_2}(P,\DD)}{q_2^{\alpha(P)}}q_1^{\gamma(R)}-\frac{a_{q_1q_2}(P,\DD)}{q_2^{\alpha(P)}}\leq \lambda\left(\frac{S(q_1)_P}{P^{[q_2]}S(q_1)_P}\right)\leq CC_1|\mathcal{S}(R)|q_1^{\gamma(R)}q_2^{\heigh(P)-1}.$$ Dividing both sides of the inequality by $q_1^{\gamma(R)}q_2^{\heigh(P)-1}$ shows that $$s_{q_2}(P,\DD)-s_{q_1q_2}(P,\DD)\leq \frac{CC_1|\mathcal{S}(R)|}{q_2}.$$ Letting $q_1\rightarrow \infty$, and relabeling constants, shows that there is a constant $C$, independent of $P,q$ such that $$s_q(P,\DD)-s(P,\DD)<\frac{C}{q}.$$ 

\medskip

To see that $s(-,\DD)$ is lower semi-continuous at $P\in \Spec(R)$ we may assume $s(P,\DD)>0$, else $s(-,\DD)$ is trivially lower semi-continuous. Thus we may assume that $R_P$ is a strongly F-regular domain. In particular, $s(-,\DD)$ is the limit of lower semicontinuous functions in an open neighborhood of $P$. Thus the lower semicontinuity of the $s_q(-,\DD)$ will now imply the lower semicontinuity of $s(-,\DD)$ since there is a constant $C$ independent of $Q\in \Spec(R)$ such that $s_q(Q,\DD)-s(Q,\DD)<\frac{C}{q}$. Observe that up to this point in the proof we have not used the assumption that the Cartier subalgebra $\DD$ satisfies condition (\ref{star}).

\medskip

  Now assume that the Cartier subalgebra $\DD$ satisfies condition (\ref{star}). Let $C$ and $\mathcal{S}(R)$ be as in Lemma \ref{Global version of Dutta's Lemma} applied to the F-finite domain $R$. Let $S(q_1)$ be the cokernels of the inclusions $R^{q_1^{\gamma(R)}}\xrightarrow{f_{q_1}} R^{1/q_1}$. Then there are short exact sequences $$0\rightarrow R^{q_1^{\gamma(R)}}\xrightarrow{f_{q_1}} R^{1/q_1}\rightarrow S(q_1)\rightarrow 0.$$ We claim that $f_{q_1}(I^\DD_{q_2}(P)R^{q_1^{\gamma(R)}})\subseteq I^\DD_{q_1q_2}(P)^{1/q_1}.$ Let $x\in R^{q_1^{\gamma(R)}}$ and $r\in I^\DD_{q_2}(P)$. Then $f_{q_1}(rx)^{q_1}=r^{q_1}f_{q_1}(x)^{q_1}\in I^\DD_{q_2}(P)^{[q_1]}\subseteq I^\DD_{q_1q_2}(P)$ by part (3) of Lemma \ref{BSTLemma} and the assumption (\ref{star}). Therefore there are induced exact sequences $$\frac{R_P^{q_1^{\gamma(R)}}}{I_{q_2}^\DD(P)R_P^{q_1^{\gamma(R)}}}\xrightarrow{f_{q_1}} \frac{R_P^{1/q_1}}{I_{q_1q_2}^{\DD}(P)^{1/q_1}}\rightarrow \tilde{S}(q_1)\rightarrow 0.$$ Observe that by part (2) of Lemma \ref{BSTLemma} that $P^{[q_2]}R_P$ kills $R_P^{1/q_2}/I_{q_1q_2}^\DD(P)^{1/q_1}$, hence $P^{[q_1]}R_P$ kills $\tilde{S}(q_1)$. Therefore $\tilde{S}(q_1)$ is the homomorphic image of $S(q_1)/P^{[q_2]}S(q_1)_P$. We can now proceed as before to get a constant $C$ independent of $P$ and $q$ such that $$s(P,\DD)-s_q(P,\DD)<\frac{C}{q}.$$ Hence there is a constant $C$ independent of $P\in \Spec(R)$ such that $|s(P,\DD)-s_q(P,\DD)|< \frac{C}{q}$, which implies $s_q(-,\DD)$ converges uniformly to $s(-,\DD)$.

\end{proof}

\begin{center}{Acknowledgements} \end{center}

The author would like to Ian Aberbach for his support during the duration of this project and suggesting the problem, Kevin Tucker for his suggestions which made the proofs much simpler, Yongwei Yao for carefully reading an earlier version and providing feedback, and Dale Cutkosky for his comments.

\bibliography{bibliography}{}
\bibliographystyle{plain}

\end{document}